\documentclass[11pt]{article}
\usepackage{amssymb, amsmath}
\usepackage[mathscr]{eucal}
\usepackage{graphics}
 \usepackage{amsmath,amsthm, amsfonts,amssymb} 
    \usepackage[mathscr]{eucal}
    \usepackage{verbatim}
    \usepackage{cases}
    \usepackage{graphicx}
    \usepackage{float}
    \usepackage{url}
    \usepackage{setspace} 
    \usepackage[font=small,format=plain,labelfont=bf,up,textfont=it,up]{caption}

    \newtheorem{theorem}{Theorem}[section]
    \newtheorem{lemma}{Lemma}[section]

    \newtheorem{definition}{Definition}[section]
    \newtheorem{corollary}{Corollary}[section]

    \numberwithin{equation}{section}

    \newenvironment{remark}[1][Remark.]{\begin{trivlist}
    \item[\hskip \labelsep {\bfseries #1}]}{\end{trivlist}}

    \numberwithin{equation}{section}
    \numberwithin{figure}{section}

    \newcommand{\R}{\mathbb{R}}

    \newcommand{\Var}{\mathrm{Var}}
    \newcommand{\Cov}{\mathrm{Cov}}

    \newcommand{\note}[1]
    {$^{(!)}$\marginpar[{\hfill\tiny{\sf{#1}}}]{\tiny{\sf{(!) #1}}}}

    \newcommand{\newchapter}[3] 
	{                           
        \chapter[#2]{#3}
        \chaptermark{#1}
        \thispagestyle{myheadings}
	}
\allowdisplaybreaks
\begin{document}

\begin{titlepage}
\title{Extrema of locally stationary Gaussian fields on growing manifolds\\[20pt]}
%
\author{Wanli Qiao\\
    Department of Statistics\\
     University of California\\
    One Shields Ave.\\
    Davis, CA 95616-8705\\
    USA
\and
    Wolfgang Polonik\\
    Department of Statistics\\
    University of California\\
    One Shields Ave.\\
    Davis, CA 95616-8705\\
    USA}
\date{\today}
\maketitle
\vspace*{-0.8cm}

\begin{abstract} \noindent We consider a class of non-homogeneous, continuous, centered Gaussian random fields $\{X_h(t), t \in {\cal M}_h;\,0 < h \le 1\}$ where ${\cal M}_h$ denotes a rescaled smooth manifold, i.e. ${\cal M}_h = \frac{1}{h} {\cal M},$ and study the limit behavior of the extreme values of these Gaussian random fields when $h$ tends to zero, which means that the manifold is growing. Our main result can be thought of as a generalization of a classical result of Bickel and Rosenblatt (1973a), and also of results by Mikhaleva and Piterbarg (1997). \\[5pt]

\end{abstract}
\vfill
{\small This research was partially support by the NSF-grant DMS 1107206\\
{\em AMS 2000 subject classifications.} Primary 60G70, 60G15.\\
{\em Keywords and phrases.} Local stationarity, extreme values, triangulation of manifolds}

\end{titlepage}
\section{Introduction} \label{intro}
Extreme value behavior of Gaussian processes is an important topic in probability theory and a crucial ingredient to many statistical inference procedures. See for instance Chernozhukov et al.  (2014) for a very recent, general  contribution to this topic. Here we are considering the extreme value behavior of Gaussian fields on manifolds, which plays an important role in statistical inference. In fact  there is recent growing interest in the statistical literature in inference for manifolds such as integral curves (Koltchinskii et al., 2007), ridges or filaments (Hall et al., 1992, Genovese et al., 2012a, 2014, Chen et al., 2013, 2014a, 2014b, Qiao and Polonik, 2015a), level sets (or level curves) (Lindgren et al., 1995, Cuevas et al., 2006, Chen et al., 2015, Qiao and Polonik,  2015b), or the boundaries of the support of a probability density function (Cuevas et al., 2004, Biau et al., 2008). In applications these objects correspond to various types of geometric objects in geoscience (fault lines), astronomy (cosmic web) or neuroscience (fibers). \\

The derivation of asymptotic distributional results for such types of geometric objects often poses interesting technical challenges. In the case of density ridges in ${\mathbb R}^2$, Qiao and Polonik (2015a) were able to derive such a distributional result for a smoothing based plug-in estimator. Note that density ridges in ${\mathbb R}^2$ can be considered as curves indexed by a parameter, and the distributional result in Qiao and Polonik (2015a) is {\em uniform} over the parameter. Because of the pointwise asymptotic normality of the estimator and the uniform nature of the result, the extreme value behavior of a Gaussian process on growing (rescaled) manifold comes into play there. The corresponding result needed in Qiao and Polonik (2015a) is proved here. In fact we derive a much more general result, allowing for arbitrary dimension and more general non-stationarity.\\

Also in other statistical literature the construction of uniform confidence bands for a target quantity rely on the asymptotic excursion probability of Gaussian processes or fields on growing sets. See, for instance, Bickel and Rosenblatt (1973a),  Konakov and Piterbarg (1984), Rio (1994), Gin\'{e} et al. (2003), Sharpnack and Arias-Castro (2014).  In this literature, the goal becomes to find the  excursion probability of Gaussian random processes or fields over growing sets in the form of 
\begin{align}\label{GeneralForm}
\lim_{h\rightarrow 0}\mathbb{P}\left(a_h\left(\sup_{x\in\mathbf{E}/h}X_h(x) - b_h\right)\leq u\right),\quad u \in {\mathbb R},
\end{align}
where $X_h$ for each $h$ is a Gaussian process or field, and $\mathbf{E}$ is original set on which the estimation is constrained, and $\mathbf{E}/h = \big\{z: hz \in \mathbf{E}  \big\}$ and $h \in {\mathbb R}$ is a parameter. The quantities $a_h$ and $b_h$ need to be determined so that this limit is non-degenerate. \\

For instance, Bickel and Rosenblatt (1973a) derive the asymptotic distribution of the quantity $\sup_{x \in [0,1]}|\widehat{f}_n(x) - f(x)|,$ where $\widehat{f}_n$ is a kernel density estimator based on a sample of $n$ independent observations from $f$. In this case, $\mathbf{E}$ is a compact set $[0,1]$, and the parameter $h$ is the bandwidth of the kernel density estimator, so that $\mathbf{E}/h = [0,\frac{1}{h}]$ which grows to the positive real line as $h \to 0.$ The case of a multivariate kernel density estimator was treated later in Rosenblatt (1976), where $\mathbf{E} = [0,1]^d$ and thus $\mathbf{E}/h = [0,1/h]^d$, growing to the positive quadrant for $h \to 0$. The corresponding derivations rely heavily on an approximation of  $\widehat{f}_n(x) - f(x), x \in [0,1],$ by Gaussian processes. A similar idea underlies the derivations in Qiao and Polonik (2015a). There, however, the processes in question are more complex. The set $\mathbf{E}$ is a certain manifold (ridge line) and a uniform nonparametric confidence region of $\mathbf{E}$ itself is of interest. It turns out that in order to achieve this, the main task is to find the limit of (\ref{GeneralForm}) with $\mathbf{E}$ a smooth manifold, and this  type of problem is considered below in a general set-up.\\


From the perspective of probability literature, excursion probability of Gaussian processes and fields is a classical and important topic. For the case of $\mathbf{E}$ being an  interval or a (hyper-) cube, the asymptotic distribution in (\ref{GeneralForm}) was studied in Pickands (1969a), Berman (1982), Leadbetter et al. (1983), Seleznjev (1991), Berman (1992), Seleznjev (1996), H\"{u}sler (1999), H\"{u}sler et al. (2003), Seleznjev (2006), Tan et al. (2012), and Tan (2015). In this literature the Hausdorff dimension of $\mathbf{E}$ is the same as the one of the ambient space for $x$.\\

Excursion probability for Gaussian random fields over some more general (but fixed) parameter set is a classical subject and widely studied, e.g. Adler (2000), Adler and Taylor (2007), Aza\"{i}s and Wschebor (2009). In particular, in Piterbarg (1996), Mikhaleva and Piterbarg (1997), and Piterbarg and Stamatovich (2001) excursion probabilities of Gaussian random fields over fixed manifolds can be found. In a recent work the excursion probability of Gaussian fields with some (local) isotropic properties indexed on a manifold is revisited in Cheng (2015).\\
 
This paper is deriving a result of type (\ref{GeneralForm}), where the underlying process $\{X_h(t), t \in {\cal M}_h;\,0 < h \le 1\}$ is a non-homogeneous, continuous, centered Gaussian random field, and ${\cal M}_h$ denotes a rescaled smooth, compact manifold. Our main result can be considered as a generalization of the classical Bickel and Rosenblatt result discussed above, as well as of a result by Mikhaleva and Piterbarg (1997) who considered a {\em fixed} manifold. Our proof combines ideas from both Bickel and Rosenblatt (1973a) and Mikhaleva and Piterbarg (1997).\\

The detailed set-up is as follows. Let $r,n\in\mathbb{Z}^+$ with $n\geq 2$ and $1\leq r< n$. Let $\mathcal{H}_1\subset\mathbb{R}^n$ be a compact set and $\mathcal {M}_1\subset\mathcal{H}_1$ be a $r$-dimensional Riemannian manifold with ``bounded curvature'', the explicit meaning of which will be made clear later. For $0 < h \le 1$ let $\mathcal{H}_h:=\{t:ht\in\mathcal{H}_1\}$ and $\mathcal {M}_h:=\{t:ht\in\mathcal{M}_1\}$. Further let
\begin{align}\label{proc}
\{X_h(t), t \in {\cal M}_h;\,0 < h \le 1\}
\end{align}
denote a class of non-homogeneous, continuous, centered Gaussian fields indexed by ${\cal M}_h, 0 < h \le 1.$ Our goal is to derive conditions assuring that for each $z > 0$ we can construct $\theta_h(z)$ with
\begin{align*}
\lim_{h\rightarrow0}\mathbb{P}\Big\{\sup_{t\in\mathcal {M}_h}|X_h(t)|\leq \theta_h(z) \Big\}=\exp\{-2\exp\{-z\}\}.
\end{align*}
%

\begin{section}{Main Result}\label{MainResult}

First we define the notion of local stationarity used here. The first definition can be found in Mikhaleva and Piterbarg (1997), for instance.\\[-5pt]

\begin{definition}[Local $(\alpha, D_t)$-stationarity] A non-homogeneous random field $X(t),t\in \mathcal {S}\subset\mathbb{R}^n$  is locally $(\alpha, D_t)$-stationary, if the covariance function $r(t_1,t_2)$ of $X(t)$ satisfies the following property. For any $s\in \mathcal {S}$ there exists a non-degenerate matrix $D_s$ such that for any $\epsilon>0$ there exists a positive $\delta(\epsilon)$ with
\begin{align*}
1-(1+\epsilon)\|D_s(t_1-t_2)\|^\alpha\leq r(t_1,t_2)\leq1-(1-\epsilon)\|D_s(t_1-t_2)\|^\alpha
\end{align*}
for $\|t_1-s\|<\delta(\epsilon)$ and $\|t_2-s\|<\delta(\epsilon)$.
\end{definition}
Observe that this definition in particular says that ${\rm Var}(X(t)) = 1$ for all $t$. Since here we are considering random fields indexed by $h$ and study their behavior as $h \to 0$, we will need local $(\alpha, D_t)$-stationarity to hold in a certain sense uniformly in $h$. The following definition makes this precise.

\begin{definition}[Local equi-$(\alpha, D_t)$-stationarity] Consider a class of non-homogeneous random fields $X_h(t),t\in \mathcal {S}_h\subset\mathbb{R}^n$ indexed by $h\in \mathbb{H}$ where $\mathbb{H}$ is an index set. We say $X_h(t)$ is locally equi-$(\alpha, D_t^h)$-stationary, if the covariance function $r_h(t_1,t_2)$ of $X_h(t)$ satisfies the following property. For any $s\in \mathcal {S}_h$ there exists a non-degenerate matrix $D_s^h$ such that for any $\epsilon>0$ there exists a positive $\delta(\epsilon)$ independent of $h$ such that 
\begin{align*}
1-(1+\epsilon)\|D_s^h(t_1-t_2)\|^\alpha\leq r_h(t_1,t_2)\leq1-(1-\epsilon)\|D_s^h(t_1-t_2)\|^\alpha
\end{align*}
for $\|t_1-s\|<\delta(\epsilon)$ and $\|t_2-s\|<\delta(\epsilon)$.
\end{definition}

An example for such a class of Gaussian random fields (with $n = \alpha = 2$) is provided by the fields introduced in Qiao and Polonik (2015a) - see (\ref{GaussianExample}) below. \\ 

We also need the concept of a condition number of a manifold (see also Genovese et al., 2012b). For an $r$-dimensional smooth manifold $\mathcal{M}$ embedded in $\mathbb{R}^n$ let $\Delta(\mathcal{M})$ be the largest $\lambda$ such that each point in $\mathcal{M}\oplus \lambda$ has a unique projection onto $\mathcal{M}$, where $\mathcal{M}\oplus \lambda$ denotes the $\lambda$-enlarged set of $\mathcal{M}$, i.e. the union of all open balls of radius $\lambda$ and midpoint in ${\mathcal M}$. $\Delta(\mathcal{M})$ is called \textit{condition number} of $\mathcal{M}$ in some literature. A compact manifold embedded in a Euclidean space has a positive condition number, see de Laat (2011), and references therein. A positive $\Delta(\mathcal{M})$ indicates a ``bounded curvature'' of $\mathcal{M}$. As indicated in Lemma 3 of Genovese et al. (2012b), on a manifold with a positive condition number, small Euclidean distance implies small geodesic distance.\\


Now we state the main theorem of this section. It is an result about the asymptotic behavior of the extreme values of locally equi-$(\alpha, D_t)$-stationary continuous Gaussian random fields indexed by a parameter $h$ as $h \to 0$.  As indicated above, our result generalizes Theorem \ref{Piterbarg} in Piterbarg and Stamatovich (2001) and Theorem A1 in Bickel and Rosenblatt (1973a). \\

For an  $n \times r$ matrix $G$ we denote by $\|G\|^2_r$ the sum of squares of all minors of order $r$, and $H_{\alpha}^r$ denotes the generalized Pickands constant (see section~\ref{misc} for a definition). At each $u\in\mathcal{M}$ let $T_u\mathcal{M}$ denote the tangent space at $u$ to $\mathcal{M}$ and let $T_u^\perp\mathcal{M}$ be the normal space, which is a $n-r$ dimensional hyperplane.
\\

\begin{theorem}\label{ProbMain}
Let $\mathcal{H}_1\subset\mathbb{R}^n$ be a compact set and $\mathcal{H}_h:=\{t:ht\in\mathcal{H}_1\}$ for $0<h\leq1$. Let $\{X_h(t), t\in \mathcal{H}_h, 0<h\leq1\}$ be a class of Gaussian centered locally equi-$(\alpha, D_t^h)$-stationary fields with $D_t^h$ continuous in $h\in (0, 1]$ and $t\in \mathcal{H}_h$. Let $\mathcal {M}_1\subset\mathcal{H}_1$ be a $r$-dimensional compact Riemannian manifold with $\Delta(\mathcal{M}_1)>0$ and $\mathcal {M}_h:=\{t:ht\in\mathcal{M}_1\}$ for $0<h\leq1$. Suppose that $\lim_{h\rightarrow0,ht=t^*}D_t^h=D_{t^*}^0$ uniformly in $t^*\in \mathcal{H}_1$, where all the components of $D_{t^*}^0$ are continuous and uniformly bounded in $t^*\in \mathcal{H}_1$. Further assume the existence of positive constants $C$ and $C^\prime$ such that
\begin{align}\label{CCprime}
0<C\leq \inf\limits_{\substack{0<h\leq 1, hs\in \mathcal{H}_1 \\ t\in \mathbb{R}^n\backslash \{0\}}} \frac{\|D_s^h\;t\|^\alpha}{\|t\|^\alpha}\leq \sup\limits_{\substack{0<h\leq 1, hs\in \mathcal{H}_1 \\ t\in \mathbb{R}^n\backslash \{0\}}} \frac{\|D_s^h\;t\|^\alpha}{\|t\|^\alpha}\leq C^\prime<\infty.
\end{align}
%
%
%
%
%
%
For any $\delta>0$, define
\begin{align*}
Q(\delta):=\sup_{0<h\leq 1}\{|r_h(x+y,y)|: x+y\in\mathcal {M}_h, y\in\mathcal {M}_h, \|x\|>\delta\}
\end{align*}
where $r_h$ is the covariance function of $X_h(t)$. Suppose for any $\delta>0$, there exists a positive number $\eta$ such that
\begin{align}\label{SupGauss1}
Q(\delta) < \eta<1,
\end{align}
In addition, assume that there exist a function $v(\cdot)$ and a value $\delta_0>0$ such that for any $\delta>\delta_0$
\begin{align}\label{SupGauss2}
Q(\delta)\Big|[\log(\delta)]^{2r/\alpha}\Big|\leq v(\delta).
\end{align}
where $v$ is a monotonically decreasing function with $v(a^p)=O(v(a))=o(1)$ and $a^{-p}=o(v(a))$ as $a\rightarrow\infty$ for any $p>0$. For any fixed $z$, define
\begin{align}\label{ThetaExp}
\theta \equiv \theta(z) &=\sqrt{2r\log{h^{-1}}}+\frac{1}{\sqrt{2r\log{h^{-1}}}}\bigg[z+\Big(\frac{r}{\alpha}-\frac{1}{2}\Big)\log{\log{h^{-1}}}\nonumber\\
&\hspace{3cm}+\log\bigg\{\frac{(2r)^{r/\alpha-1/2}}{\sqrt{2\pi}}H_\alpha^{(r)}\int_{\mathcal {M}_1}\|D_s^0 M_s^1\|_rds\bigg\}\bigg],
\end{align}
where $M_s^1$ is a $n\times r$ matrix with orthonormal columns spanning $T_s\mathcal {M}_1.$ Then
\begin{align*}
\lim_{h\rightarrow0}\mathbb{P}\Big\{\sup_{t\in\mathcal {M}_h}|X_h(t)|\leq \theta \Big\}=\exp\{-2\exp\{-z\}\}.
\end{align*}
\end{theorem}

{\bf Remarks.}\hfill
\begin{enumerate}
\item Note that with (\ref{CCprime}), local equi-$(\alpha, D_t^h)$-stationarity is equivalent to
\begin{align*}
r_h(t_1,t_2)=1-\|D_s^h(t_1-t_2)\|^\alpha+o(\|t_1-t_2\|^\alpha) \quad \textrm{as} \quad \|t_1-t_2\|\rightarrow 0,
\end{align*}
uniformly for $t_1,t_2\in\mathcal{H}_h$ and uniformly in $h$.
\item An example of a function $v(\delta)$ satisfying the properties of $v(u)$ required in the theorem is given by $v(\delta) = \log(\delta)^{-\beta}$ for $\beta>0$. 
\item Qiao and Polonik (2015a) use a special case of the above theorem. In that paper a 1-dimensional growing manifold $\mathcal{M}_h$  embedded in $\mathbb{R}^2$ was considered. The Gaussian random field of interest there is
\begin{align}\label{GaussianExample}
U_h(x)=a_1(hx)\int \big(A_1(hx))^T  d^2   K(x-s) dW(s),
\end{align}
where $W$ is a 2-dimensional Wiener process,  $A_1: \mathbb{R}^2\mapsto \mathbb{R}^3$ and $a_1: \mathbb{R}^2\mapsto \mathbb{R}$ are smooth functions, $K: \mathbb{R}^2\mapsto \mathbb{R}$ is a smooth kernel density function with the unit ball in $\mathbb{R}^2$ as its support, and $d^2$ is an operator such that $d^2f(x) = \left( f^{(2,0)}(x),\,f^{(1,1)}(x),\,f^{(0,2)}(x)\right)^T$ for any twice differentiable function $f: \mathbb{R}^2 \mapsto \mathbb{R}$. It is shown in Qiao and Polonik (2015a) that the assumptions formulated in that paper insure that the processes $U_h(x)$ satisfy the assumptions of our main theorem in the special case of $r = 1$, $n= 2, \alpha = 2$, and $Q(\delta) = 0$ for $\delta > \delta_0$. This in particular means that the function $v(\delta) =  \log(\delta)^{-\beta}$ for $\beta>0$ works in this case. The fact that this special function $Q(\delta)$ can be used there follows from the assumption that the support of $K$ (and its second order partial derivatives) is bounded. This implies that the covariances of $U_h(x_1)$ and $U_h(x_2)$ become zero once the distance $\|x_1 - x_2\|$ exceeds a certain threshold. The derivation of the corresponding matrices $D_s$ can also be found in Qiao and Polonik (2015a).

%
\end{enumerate}
%

\section{Proof of Theorem~\ref{ProbMain}}\label{proof}
The following notation and definitions are used below. Given a set $U$ and a metric $d_U$ on $U$, a set $S\subset U$ is an $\epsilon$-net if for any $u\in U$, we have $\inf_{s\in S}d_U(s,u)\leq \epsilon$ and for any $s,t\in S$, we have $d_U(s,t)\geq \epsilon$. Let further $\phi(u)=\frac{1}{\sqrt{2\pi}}e^{-\frac{u^2}{2}},\,\Phi(u)=\int_{-\infty}^u\phi(v)dv$, $\bar\Phi(u)=1-\Phi(u)$ and $\Psi(u)=\phi(u)/u$. Let $\|\cdot\|$ be the Euclidean norm. We also let denote $V_r$ denote $r$-dimensional Hausdorff measure.\\

The proof is constructing various approximations to $\sup_{t\in\mathcal {M}_h}|X_h(t)|$ that will facilitate the control of the probability $\mathbb{P}(\sup_{t\in\mathcal {M}_h}|X_h(t)| \le \theta)$. Essentially the process $X_h(t)$ on the manifold is linearized by first approximating the manifold locally via tangent planes, and then defining an approximating process on these tangent planes. This idea underlying the proof is typical for deriving extreme value results for such processes (e.g. see H\"{u}sler et al. 2003). We begin with some preparations thereby outlining the main ideas of the proof.\\

{\bf (i)} {\em Partitioning ${\cal M}_h$:} We partition the manifold ${\cal M}_h$ as follows. Suppose that $V_r(\mathcal {M}_1) = \ell$ so that $V_r(\mathcal {M}_h) = \ell/h^r$. For a fixed $\ell^*< \ell$, there exists an $\ell^*$-net on $\mathcal{M}_h$ with respect to geodesic distance with cardinality of $O((h\ell^{*})^{-r})$. A Delaunay triangulation using the $\ell^*$-net results in a partition of $\mathcal {M}_h$ into $m_h=O((h\ell^{*})^{-r})$ disjoint pieces $\{J_{k,m_h}: k=1,2,\cdots,m_h\}.$ The construction is such that $\max_{k = 1,\ldots,m_h}V_r(J_{k,m_h})$, the norm of this partition, is $O({\ell^*}^r)$, uniformly in $h$. It is known that for any $r \in {\mathbb Z}^+$ with $r<n$ (and for $\ell^*$ small enough) such an $\ell^*$-net and a Delaunay triangulation exist for compact Riemannian manifolds (see e.g. de Laat 2011). (In the case of $r=1$, the construction just described simply amounts to choosing all the $O(1/h\ell^*)$ many sets $J_{k,m_h}$ as pieces on the curve ${\cal M}_h,$ which has length at most $\ell^*$.) One should point out that while $\ell^*$ has to be chosen sufficiently small, it is a constant not depending on $h$. In particular this means that it does not tend to zero in this work.  \\

{\bf (ii)} {\em `Small blocks - large blocks' approach:} For sufficiently small $\delta>0$, let $\mathcal {M}_h^{-\delta}\subset\mathcal {M}_h$ be the $\delta$-enlarged neighborhood (using geodesic distance) of the union of the boundaries of all $J_{k,m_h}$. The minus sign in the superscript indicates that this (small) piece will be `cut out' in the below construction.  We obtain $J_{k,m_h}^\delta = J_{k,m_h}\backslash \mathcal {M}_h^{-\delta}$ (`large blocks') and $J_{k,m_h}^{-\delta} = J_{k,m_h}\backslash J_{k,m_h}^\delta$ (`small blocks') for $1\leq k\leq m_h$. Geometrically we envision $J_{k,m_h}^{-\delta}$ as a small strip along the boundaries of $J_{k,m_h}$ (lying inside  $J_{k,m_h}$), and $J_{k,m_h}^\delta$ is the set that remains when  $J_{k,m_h}^{-\delta}$  is cut out of $J_{k,m_h}$. We have $V_r(J_{k,m_h}^{-\delta}) = O(\delta)$, uniformly in $k$ and $h$. The construction of the partition is such that the boundaries of the projections of all the sets $J_{k,m_h},\,J_{k,m_h}^\delta$ and $J_{k,m_h}^{-\delta}$ onto the local tangent planes are null sets, and thus Jordan measurable. This will be used below.\\

Let $B_h(A) = \{ \sup_{t\in A}|X_h(t)| \ge \theta\}$ and as a shorthand notation we use $p_h(A) = \mathbb{P}(B_h(A)).$ \\

Approximating $\mathcal{M}_h$ by $\bigcup\limits_{k\leq m_h}J_{k,m_h}^\delta$ leads to a corresponding approximation of $p_h(\mathcal{M}_h)$ by  $p_h(\bigcup\limits_{k\leq m_h}J_{k,m_h}^\delta)$. Even though the volume of $\bigcup\limits_{k\leq m_h}J_{k,m_h}^{-\delta},$ i.e. the difference of $\mathcal{M}_h$ and $\bigcup\limits_{k\leq m_h}J_{k,m_h}^\delta$, tends to infinity as $h\rightarrow0$ if we consider $\delta$ fixed, the difference $p_h(\mathcal{M}_h) - p_h(\bigcup\limits_{k\leq m_h}J_{k,m_h}^\delta)$ turns out to be of the order $O(\delta)$. Thus we have to choose $\delta$ small enough.\\

{\bf (iii)} {\em Refinement of the partition:} Let $J$ denote one of the sets $J_{k,m_h},\,J_{k,m_h}^\delta$ and $J_{k,m_h}^{-\delta}$, and let $\{S_i^h(J)\subset J, i=1, \cdots, N_h(J)\}$ be a cover of this piece constructed using the same Delaunay triangulation technique as above, but of course based on a smaller mesh. As above, by controlling the mesh size, we can control the norm of the partition uniformly over $h$, because of the uniform boundedness of the curvatures of the manifolds ${\cal M}_h$. \\

The probabilities $p_h(J_{k,m_h}^\delta)$ are approximated by the sum of the probabilities $p_h(S_i^h)$, with $S_i^h$ the cover of $J_{k,m_h}^\delta$ introduced above. It will turn out that the approximation error can be bounded by a double sum, using Bonferroni inequality. To show this double sum is negligible compared with the sum, we have to make sure the volume of $S_i^h$ is not too small, as long as $S_i^h$ is sufficiently small. The double sum will be small if $h$ is small. It turns out that $p_h(J_{k,m_h}^\delta)$ essentially behaves like the tail of a normal distribution.\\

{\bf (iv)} {\em Projection into tangent space of refined partition:} We approximate the small pieces $S_i^h$ on the manifold by $\tilde{S}_i^h$, the projection onto the tangent space and correspondingly approximate the probabilities $p_h(S_i^h)$ by the corresponding probabilities of a transformed field over $\tilde{S}_i^h$. More precisely, we choose some point $s_i^h$ on $S_i^h(J)$ and orthogonally project $S_i^h(J)$ onto the tangent space of $\mathcal {M}_h$ at the point $s_i^h$. We denote the mapping by $P_{s_i^h}(\cdot)$ or simply $P_{s_i}(\cdot)$ and we let $\tilde S_i^h(J)=P_{s_i}(S_i^h(J))$, which, as indicated above are Jordan measurable by construction. The error generated from the approximation is controlled by choosing the norm of the partitions given by the $S_i^h$ to be sufficiently small.  \\

In the following, if $J$ is explicitly indicated in the context, then we often drop $J$ in the notation and simply write $S_i^h$ instead of $S_i^h(J)$. For simplicity and generic discussion, we sometimes also omit the index $i$ of $s_i^h$, $S_i^h$ and $\tilde S_i^h$.\\
\begin{figure}[h]
\centering
\includegraphics[width=10cm]{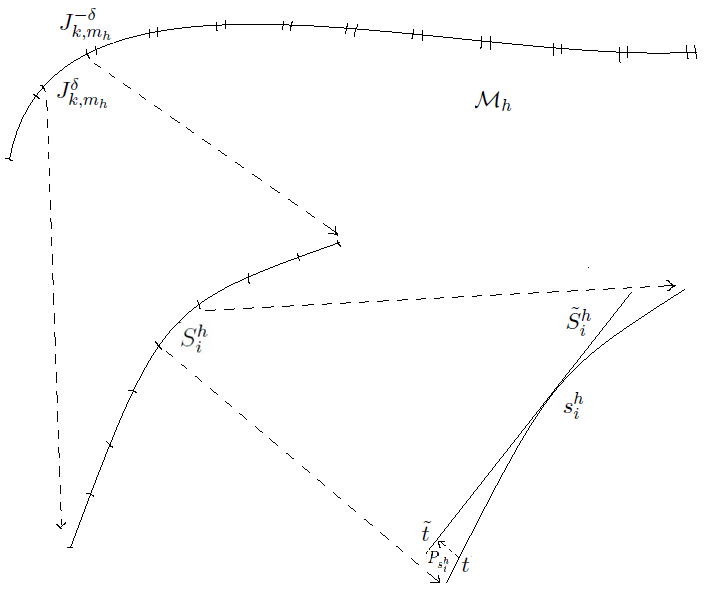}
\caption{This figure visualizes some of the definitions introduced here in the case $r=1$ and $n = 2$.}
\label{Figure1}
\end{figure}

{\bf (v)} {\em Discretizing the projection into the tangent space:}  The probabilities  $p_h(\tilde{S}_i^h)$ of the sets $\tilde{S}_i^h$ introduced in the previous step are approximated by replacing the probability of the supremum in $p_h(\tilde{S}_i^h)$ by a maximum over a collection of `dense grid' on $\tilde{S}_i^h$. The accuracy of the approximation is controlled by choosing both $\gamma$ and $h$ sufficiently small. The construction of the dense grid is as follows:\\

Let $\{M_{s^h}^j: j=1,\cdots,r\}$ be linearly independent orthonormal vectors spanning the tangent space of $\mathcal {M}_h$ at the point $s^h$, and let $M_s^h$ denote the $n\times r$ matrix with $M_{s^h}^j$ as columns. For a given $\gamma$ consider the (discrete) set $\tilde\Gamma_{\gamma \theta^{-2/\alpha}}(\tilde S^h):=\{u: u=s^h+\sum_{j=1}^ri_j\gamma \theta^{-2/\alpha}M_{s^h}^j\in\tilde S^h, i_j\in\mathbb{Z}\}$ and let $\Gamma_{\gamma \theta^{-2/\alpha}}(S^h)=(P_{s^h})^{-1}(\tilde\Gamma_{\gamma \theta^{-2/\alpha}}(\tilde S^h))$, which is a subset of $S^h$. Note that the geodesic distance between any two adjacent points in $\Gamma_{\gamma \theta^{-2/\alpha}}(S^h)$ is still of the order $O(\gamma \theta^{-2/\alpha}),$ again due to the assumed uniformly positive condition number of the manifolds ${\cal M}_h$.\\

The collection of all sets of dense points in $\Gamma_{\gamma \theta^{-2/\alpha}}(S^h)$ results in a set ${\mathbb T}^\delta_h$ of dense points in $\bigcup_{k\leq m_h}J_{k,m_h}^\delta$. It will turn out that the probability $1 - p_h({\mathbb T}^\delta_h) = \mathbb{P}\big(\bigcap_k \big(B_h({\mathbb T}^\delta_h \cap J_{k,m_h}^\delta)\big)^c\big)$ can be approximated by assuming the events $\big(B_h({\mathbb T}^\delta_h \cap J_{k,m_h}^\delta)\big)^c,\,k = 1,\ldots,m_h,$ to be independent. To make sure this approximation is valid, $\delta$ may be not too small and $\gamma$ may be not too small compared with $h$.\\[5pt]

Putting everything together will then complete the proof. \\

{\bf Details of the proof.} We now present the details by using the notation introduced above. We split the proof into different parts in order to provide more structure. Note that the parts do not really follow the logical steps outlined above.\\

{\bf Part 1}. Recall the definition of the refined partition $\{S_i^h,\,i = 1,\ldots,N\}$ of $J_{k,m_h}$ given in (iii). Here we show that $\sum_i p_h(S_i^h) \approx \theta^{2r/\alpha}\Psi(\theta) \,\;H_\alpha^{(r)} \int_{J_{k,m_h}}\|D_s^hM_s^h\|_r ds$ and that a similar approximation holds for $J_{k,m_h}$ replaced by $J^\delta_{k,m_h}$. To this end we will utilize the projections $\tilde{S}_i^h$ of $S_i^h$ onto the tangent space (see (iv)) as well as the approximation of $\tilde{S}_i^h$ by a set of dense points introduced in (v).\\

The various asymptotic approximations in this step are similar to those in the proof of Theorem 1 in Mikhaleva and Piterbarg (1997), but here we consider them in the uniform sense. As indicated above, the uniform boundedness of the curvature of $\mathcal {M}_h$ can be guaranteed due to the boundedness of the curvature (positive condition number) of $\mathcal {M}_1$. For any $\epsilon_1>0$, there exists a constant $\delta_1>0$ (not depending on $h$) such that if the volumes of all $S_i^h =S_i^h(J_{k,m_h}^\delta) $ 
are less than $\delta_1$, then we have
\begin{align}\label{S1S}
1-\epsilon_1\leq\frac{V_r(\tilde S^h)}{V_r(S^h)}\leq1+\epsilon_1,
\end{align}
where $V_r(\cdot)$ is the $r$-dimensional Hausdorff measure. On $\tilde S^h$ we consider the Gaussian field defined as
\begin{align*}
\tilde X_h(\tilde t)=X_h(t),\quad  \text{with $t \in S^h$ such that} \;\tilde t= P_{s_i}(t) \in \tilde S^h.
\end{align*}

Due to local equi-$(\alpha, D_t^h)$-stationarity of $X_h(t)$, for any $\epsilon_2>0$, the covariance function $\tilde r_h(\tilde t_1,\tilde t_2)$ of the field $\tilde X_h(\tilde t)$ satisfies
\begin{align*}
1-(1+\epsilon_2/4)\|D_s^h(t_1-t_2)\|^\alpha\leq \tilde r_h(\tilde t_1,\tilde t_2)\leq1-(1-\epsilon_2/4)\|D_s^h(t_1-t_2)\|^\alpha
\end{align*}
for all $t_1, t_2\in \tilde S^h$, if the volume of $S_h$ is less than a certain threshold $\delta_2$, which only depends on $\epsilon_2$. By possibly decreasing $\delta_2$ further we also have
\begin{align*}
1-(1+\epsilon_2/2)\|D_s^h(\tilde t_1-\tilde t_2)\|^\alpha\leq \tilde r_h(\tilde t_1,\tilde t_2)\leq1-(1-\epsilon_2/2)\|D_s^h(\tilde t_1-\tilde t_2)\|^\alpha
\end{align*}
for all $\tilde t_1,\tilde t_2\in \tilde S^h$. Note that this inequality holds uniformly over all $\tilde S^h$ under consideration, due to the curvature being bounded on $\mathcal{M}_h$.\\

On $\tilde S^h$ we introduce two homogeneous Gaussian fields $X_h^+(\tilde t),X_h^-(\tilde t)$ such that their covariance functions satisfy
\begin{align*}
r_h^+(\tilde t_1,\tilde t_2)=1-(1+\epsilon_2)\|D_s^h(\tilde t_1-\tilde t_2)\|^\alpha+o(\|D_s^h(\tilde t_1-\tilde t_2)\|^\alpha)\\
r_h^-(\tilde t_1,\tilde t_2)=1-(1-\epsilon_2)\|D_s^h(\tilde t_1-\tilde t_2)\|^\alpha+o(\|D_s^h(\tilde t_1-\tilde t_2)\|^\alpha)
\end{align*}
as $\|\tilde t_1-\tilde t_2\|\rightarrow0$. Thus if  the volumes of all $S^h$ under consideration are sufficiently small then
\begin{align*}
r_h^+(\tilde t_1,\tilde t_2)\leq \tilde r_h(\tilde t_1,\tilde t_2) \leq r_h^-(\tilde t_1,\tilde t_2)
\end{align*}
holds for all $\tilde t_1,\tilde t_2\in \tilde S^h$. This can be achieved by possibly adjusting $\delta_2$ from above. Slepian's inequality in Lemma \ref{Slepian} implies that
\begin{align*}
&\mathbb{P}\bigg(\sup_{\tilde t \in \tilde S^h} X_h^-(\tilde t)>\theta\bigg)\leq \mathbb{P}\bigg(\sup_{\tilde t \in \tilde S^h} \tilde X_h(\tilde t)>\theta\bigg)\\
&\hspace{2cm}= \mathbb{P}\bigg(\sup_{\tilde t \in S^h} X_h(t)>\theta\bigg)\leq \mathbb{P}\bigg(\sup_{\tilde t \in\tilde S^h} X^+(\tilde t)>\theta\bigg),
\end{align*}
and that
\begin{align}
&\mathbb{P}\bigg(\max_{\tilde t\in\tilde\Gamma_{\gamma \theta^{-2/\alpha}}(\tilde S^h)}X_h^-(\tilde t)>\theta\bigg)\leq \mathbb{P}\bigg(\max_{\tilde t\in\tilde\Gamma_{\gamma \theta^{-2/\alpha}}(\tilde S^h)}\tilde X_h(\tilde t)>\theta\bigg)\nonumber\\
&\hspace{2cm}=\mathbb{P}\bigg(\max_{t\in\Gamma_{\gamma \theta^{-2/\alpha}}(S^h)} X_h(t)>\theta\bigg)\leq \mathbb{P}\bigg(\max_{\tilde t\in\tilde\Gamma_{\gamma \theta^{-2/\alpha}}(\tilde S^h)}X_h^+(\tilde t)>\theta\bigg).\label{SlepianApp}
\end{align}
For $\tau\in \mathbb{R}^n$ such that $(1+\epsilon_2)^{-1/\alpha}{(D_s^h)}^{-1}\tau \in \tilde S^h$, denote $X_h^+\Big((1+\epsilon_2)^{-1/\alpha}{(D_s^h)}^{-1}\tau\Big)$ by $Y_h^+(\tau)$ as a function of $\tau$. The covariance function of $Y_h^+(\tau)$ is
\begin{align*}
r_{Y_h}^+(\tau_1,\tau_2)&=r_h^+\Big((1+\epsilon_2)^{-1/\alpha}{(D_s^h)}^{-1}\tau_1,\; (1+\epsilon_2)^{-1/\alpha}{(D_s^h)}^{-1}\tau_2\Big)\\
&=1-(1+\epsilon_2)\Big\|D_s^h\Big((1+\epsilon_2)^{-1/\alpha}{(D_s^h)}^{-1}\tau_1-(1+\epsilon_2)^{-1/\alpha}{(D_s^h)}^{-1}\tau_2\Big)\Big\|^\alpha
+o(\|\tau_1-\tau_2\|^\alpha)\\
&=1-\|\tau_1-\tau_2\|^\alpha+o(\|\tau_1-\tau_2\|^\alpha)
\end{align*}
as $\|\tau_1-\tau_2\|\rightarrow 0$.
An application of Lemma \ref{Piece} gives that for any $\epsilon_3>0$ and $\theta$ large enough
\begin{align}
&\frac{\mathbb{P}\Big(\max_{\tilde t\in\tilde\Gamma_{\gamma \theta^{-2/\alpha}}(\tilde S^h)}X_h^+(\tilde t)>\theta\Big)}{\theta^{2r/\alpha}\Psi(\theta)}\nonumber\\
&=\frac{\mathbb{P}\Big(\max_{\tau\in(1+\epsilon_2)^{1/\alpha} D_s^h\tilde\Gamma_{\gamma \theta^{-2/\alpha}}(\tilde S^h)}X_h^+\Big((1+\epsilon_2)^{-1/\alpha}{(D_s^h)}^{-1}\tau\Big)>\theta\Big)}{\theta^{2r/\alpha}\Psi(\theta)}\nonumber\\
&\leq\frac{H_\alpha^{(r)}(\gamma)}{\gamma^r}(1+\epsilon_3)V_r((1+\epsilon_2)^{1/\alpha} D_s^h\tilde S^h)\label{ApplyUniformPiece1}\\
&=(1+\epsilon_2)^{r/\alpha}(1+\epsilon_3)\frac{H_\alpha^{(r)}(\gamma)}{\gamma^r}\|D_s^hM_s^h\|_rV_r(\tilde S^h)\nonumber.
\end{align}
Similarly, by defining $Y_h^-(\tau)=X_h^-\Big((1-\epsilon_2)^{-1/\alpha}{(D_s^h)}^{-1}\tau\Big)$, we get
\begin{align}\label{ApplyUniformPiece2}
\frac{\mathbb{P}\Big(\max_{\tilde t\in\tilde\Gamma_{\gamma \theta^{-2/\alpha}}(\tilde S^h)}X_h^-(\tilde t)>\theta\Big)}{\theta^{2r/\alpha}\Psi(\theta)}\geq (1-\epsilon_2)^{r/\alpha}(1-\epsilon_3)\frac{H_\alpha^{(r)}(\gamma)}{\gamma^r}\|D_s^hM_s^h\|_rV_r(\tilde S^h).
\end{align}

Combining (\ref{S1S}), (\ref{SlepianApp}), (\ref{ApplyUniformPiece1}) and (\ref{ApplyUniformPiece2}), we obtain for $V_r(S^h)$ small enough and $\theta$ large enough that for any $\epsilon > 0$
\begin{align*}
&\big(1-\tfrac{\epsilon}{4}\big)\, \frac{H_\alpha^{(r)}(\gamma)}{\gamma^r}\|D_s^hM_s^h\|_rV_r(S^h)\\
&\hspace{2cm}\leq \frac{\mathbb{P}\Big(\max_{t\in\Gamma_{\gamma \theta^{-2/\alpha}}(S^h)} X_h(t)>\theta\Big)}{\theta^{2r/\alpha}\Psi(\theta)}\\
&\hspace{4cm}\leq \big(1-\tfrac{\epsilon}{4}\big)\,\frac{H_\alpha^{(r)}(\gamma^r)}{\gamma^r}\|D_s^hM_s^h\|_rV_r(S^h),
\end{align*}
and since Lemma \ref{HGamma} says that $\frac{H_\alpha^{(r)}(\gamma)/\gamma^r}{H_\alpha^{(r)}} \to 1$ as $\gamma \to 0$, we further have for $\gamma$ sufficiently small that 
\begin{align}
&\big(1-\tfrac{\epsilon}{2}\big)\, H_\alpha^{(r)}\|D_s^hM_s^h\|_rV_r(S^h) \nonumber\\
&\hspace{2.5cm}\leq\frac{\mathbb{P}\Big(\max_{t\in\Gamma_{\gamma \theta^{-2r/\alpha}}(S^h)} X_h(t)>\theta\Big)}{\theta^{2r/\alpha}\Psi(\theta)}\nonumber\\
&\hspace{5cm}\leq \big(1-\tfrac{\epsilon}{2}\big)\,  H_\alpha^{(r)}\|D_s^hM_s^h\|_rV_r(S^h).\label{PieceBound}
\end{align}
This in fact holds for any $S^h = S^h_i$. We now want to add over $i.$ To this end observe that  $\sum_{i=1}^{N_h}(\|D_{s_i}^hM_{s_i}^h\|_rV_r(S_i^h))$ is a Riemann sum, namely, for any $\epsilon>0$, there exists a $\delta>0$ such that for $\max_{i=1, \cdots, N_h}{V_r(S_i^h)}<\delta$, we have for $h$ sufficiently small that
\begin{align}\label{InteSum}
(1-\epsilon)\int_{J_{k,m_h}^\delta}\|D_s^hM_s^h\|_rds \leq \sum_{i=1}^{N_h}\|D_{s_i^h}^hM_{s_i^h}^h\|_rV_r(S_i^h) \leq (1+\epsilon)\int_{J_{k,m_h}^\delta}\|D_s^hM_s^h\|_rds.
\end{align}
The selection of $\delta$ only depends on $\epsilon$, and the uniformity comes from the fact that as $h\to 0$, $\|D_{t_1}^h-D_{t_2}^h\|_n=\|D_{ht_1}^0-D_{ht_2}^0\|_n+o(1)$ for any $t_1$ and $t_2$ and that $D_{t^*}^0$ is continuous in $t^*\in\mathcal{H}_1$.\\

It follows from (\ref{PieceBound}) and (\ref{InteSum}) that for any $\epsilon > 0$ and  $\gamma, \sup_{0 < h \le 1} \max_{i=1, \cdots, N_h}{V_r(S_i^h)}$ sufficiently small and $\theta$ large enough 
\begin{align}
&(1-\epsilon) H_\alpha^{(r)} \int_{J_{k,m_h}^\delta}\|D_s^hM_s^h\|_rds
\leq\frac{\sum_{i=1}^{N_h}\mathbb{P}\Big(\max_{t\in\Gamma_{\gamma \theta^{-2/\alpha}}(S_i^h)} X_h(t)>\theta\Big)}{\theta^{2r/\alpha}\Psi(\theta)} \nonumber\\
&\hspace{5cm}
\leq (1+\epsilon)\,H_\alpha^{(r)} \int_{J_{k,m_h}^\delta}\|D_s^hM_s^h\|_rds.\label{PieceSum}
\end{align}

Since the distribution of $X_h$ is symmetric, we also have
\begin{align}
&(1-\epsilon)\,H_\alpha^{(r)} \int_{J_{k,m_h}^\delta}\|D_s^hM_s^h\|_rds
\leq\frac{\sum_{i=1}^{N_h}\mathbb{P}\Big(\min_{t\in\Gamma_{\gamma \theta^{-2/\alpha}}(S_i^h)} X_h(t)<-\theta\Big)}{\theta^{2r/\alpha}\Psi(\theta)}\nonumber\\
&\hspace{5cm}\leq (1+\epsilon)\,H_\alpha^{(r)} \int_{J_{k,m_h}^\delta}\|D_s^hM_s^h\|_rds.\label{PieceSum2}
\end{align}
We emphasize that these inequalities hold when the norm of the partition is below a certain threshold that is independent of the choice of $h$.\\

Following a similar procedure as above we see that (\ref{PieceSum}) and (\ref{PieceSum2})  continue to hold (for $h$ and $\max_{i=1, \cdots, N_h}{V_r(S_i^h)}$ sufficiently small and $\theta$ large enough) if $\max_{t\in\Gamma_{\gamma \theta^{-2/\alpha}}(S_i^h)} X_h(t)$ in (\ref{PieceSum}) is replaced by $\sup_{t\in S_i^h} X_h(t)$, and similarly, $\min_{t\in\Gamma_{\gamma \theta^{-2/\alpha}}(S_i^h)} X_h(t)$ in (\ref{PieceSum2}) is replaced by $\inf_{t\in S_i^h} X_h(t)$. 
%
%
%
Moreover, if we consider $S_i^h(J_{k,m_h})$ and $S_i^h(J_{k,m_h}\backslash J_{k,m_h}^\delta)$, instead of $S_i^h(=S_i^h(J_{k,m_h}^\delta))$, these inequalities continue to hold.  In particular for $J_{k,m_h}$ we obtain  
\begin{align}\label{SumToInt}
&\frac{\sum_{i=1}^{N_h(J_{k,m_h})}\mathbb{P}\Big(\sup_{t\in S_i^h(J_{k,m_h})} X_h(t)>\theta\Big)}{\theta^{2r/\alpha}\Psi(\theta)}
=(1+ o(1))\;
%
H_\alpha^{(r)} \int_{J_{k,m_h}}\|D_s^hM_s^h\|_r ds,
\end{align}
where 
the $o(1)$-term is uniform in $1\leq k\leq m_h$ as  
$\theta \to\infty$.\\

{\bf Part 2.} Here we show that $\sum_{k\leq m_h}\mathbb{P}\Big(\sup_{t\in J_{k,m_h}}X_h(t)> \theta\Big) \approx \theta^{2r/\alpha}\Psi(\theta)H_\alpha^{(r)} \int_{\mathcal {M}_h}\|D_s^hM_s^h\|_rds$ as $h \to 0.$  Again we will use the various approximations introduced at the beginning of the proof. \\

Let $\{S_i^h : i=1, \cdots, N_h\}$ denote the partition of $J_{k,m_h}$ constructed in (iii). This partition consists of closed non-overlapping subsets, i.e. their interiors are disjoint. Let further
\begin{align*}
B_i=\Big\{\sup_{t\in S_i^h} X_h(t)>\theta\Big\}.
\end{align*}
Then obviously,
\begin{align*}
\mathbb{P}\Big(\sup_{t\in J_{k,m_h}} X_h(t)>\theta\Big)=\mathbb{P}\bigg(\bigcup_{i=1}^{N_h}B_i\bigg).
\end{align*}
We now use 
\begin{align*}
\sum_{i=1}^{N_h}\mathbb{P}(B_i)-\sum_{1\leq i<j\leq N_h}\mathbb{P}(B_i \cap B_j)\leq \mathbb{P}\bigg(\bigcup_{i=1}^{N_h}B_i\bigg) \leq\sum_{i=1}^{N_h}\mathbb{P}(B_i),
\end{align*}
and we want to show that the double sum on the left-hand side is negligible as compared to the sum, so that we essentially have upper and lower bounds for  $\mathbb{P}\big(\bigcup_{i=1}^{N_h}B_i\big)$ in terms of $\sum_{i=1}^{N_h}\mathbb{P}(B_i).$ To see this, first observe that it follows from (\ref{SumToInt}) that for $\max_{i=1,\cdots,N_h} V_r(S_i^h)$ small enough we have as $\theta\rightarrow \infty$ that
\begin{align}\label{SumB}
\sum_{i=1}^{N_h}\mathbb{P}(B_i)=O(\theta^{2r/\alpha}\Psi(\theta)).
\end{align}
We thus want to show that $\sum_{1\leq i<j\leq N_h}\mathbb{P}(B_i \cap B_j)= o(\theta^{2r/\alpha}\Psi(\theta))$ as $\theta\rightarrow \infty.$ The proof for a fixed manifold (i.e. $h$ fixed) can be found in the last part of Mikhaleva and Piterbarg (1997). Our proof for the more general case (uniformly in $h$) is following a similar procedure. It will turn out that we obtain the desired result if the norm of the partition given by the $S_i^h$ can be chosen arbitrarily small, uniformly in $h$. It has been discussed at the beginning of the proof that this is in fact the case.\\

 Let $U=\{(i,j): B_i$ and $B_j$ are adjacent$\}$ and $V=\{(i,j): B_i \text{ and } B_j \text{ are not adjacent}\}$, where non-adjacent means that their boundaries do not touch. Note that
\begin{align}\label{DoubleSum}
\sum_{1\leq i<j\leq N_h}\mathbb{P}(B_i\cap B_j)=\sum\limits_{\substack{1\leq i<j\leq N_h, \\ (i,j)\in U}}\mathbb{P}(B_i\cap B_j) + \sum\limits_{\substack{1\leq i<j\leq N_h, \\ (i,j)\in V}}\mathbb{P}(B_i\cap B_j).
\end{align}
In what follows we discuss the two sums on the right hand side of (\ref{DoubleSum}). First we consider the case that $S_i^h,\,S_j^h \in U$  are adjacent, i.e. $(i,j) \in U$. 
The developments in Part 1 are here applied to $S_i^h$, $S_{j}^h$ and $S_i^h\cup S_{j}^h$, respectively. We choose the points  where the tangent spaces are placed to be the same for $S_i^h$, $S_{j}^h$ and $S_i^h\cup S_{j}^h$, i.e., we choose this point to lie on the boundary of both $S_i^h$ and $S_{j}^h$. Simply denote this point as $s$. Then, by using the results from Part 1, for any $\epsilon>0$,  when $\max_{(i,j)\in U}V_r(S_i^h\cup S_{j}^h)$ is small enough and $\theta$ is large enough, then the bounds obtained as in Part 1 result in
\begin{align*}
\frac{\mathbb{P}(B_i \cap B_{j})}{\theta^{2r/\alpha}\Psi(\theta)}&=\frac{\mathbb{P}(B_i)+\mathbb{P}(B_j)-\mathbb{P}(B_i\cup B_{j})}{\theta^{2r/\alpha}\Psi(\theta)}\\
&\leq (1+\epsilon)H_\alpha^{(r)}\; \|D_s^hM_s^h\|_r\;V_r(S_i^h)+(1+\epsilon)H_\alpha^{(r)}\;\|D_s^hM_s^h\|_r\;V_r(S_{j}^h)\\
&\hspace{1cm}-(1-\epsilon)H_\alpha^{(r)}\; \|D_s^hM_s^h\|_r\;V_r(S_i^h\cup S_{j}^h)\\
&=2\epsilon H_\alpha^{(r)}\; \|D_s^hM_s^h\|_r\;[V_r(S_i^h)+V_r(S_{j}^h)].
\end{align*}

The sum of the right hand side of the above inequalities over $(i,j)\in U$ again is a Riemann sum that approximates an integral over $J_{k,m_h}$. Since $\lim_{h\rightarrow0,ht=t^*}D_t^h=D_{t^*}^0$ uniformly in $t^*\in\mathcal{H}_1$, and since the components of $D_{t^*}^0$ are continuous and bounded in $t^*\in \mathcal{H}_1$, there exists a finite real $c>0$ such that 
\begin{align}\label{DMbound}
\sup_{s\in\mathcal {M}_h, 0<h\leq1}\|D_s^hM_s^h\|_r\leq c.
\end{align}
Hence as $\max_{1\leq i \leq N_h}V_r(S_i^h)\rightarrow 0$ and $\theta \to \infty$, and noting that $\epsilon > 0$ is arbitrary, we have
\begin{align}\label{SubSumBound2}
\sum\limits_{\substack{1\leq i<j\leq N_h, \\ (i,j)\in U}}\mathbb{P}(B_i\cap B_j) = o(\theta^{2r/\alpha}\Psi(\theta)).
\end{align}

Next we proceed to consider the case that $(i,j) \in V$, i.e. $S_i^h, S_j^h$ are not adjacent on $J_{k,m_h}$. To find a upper bound for $\mathbb{P}(B_i \cap B_j)$, first notice that
\begin{align}\label{IntersectionBound}
\mathbb{P}(B_i \cap B_j) &= \mathbb{P}\Big(\sup_{t\in S_i^h} X_h(t)>\theta, \sup_{t\in S_j^h} X_h(t)>\theta\Big)\nonumber\\
&\leq \mathbb{P}\Big(\sup_{t\in S_i^h,s\in S_j^h} \left(X_h(t)+X_h(s)\right)>2\theta\Big).
\end{align}
In order to further estimate this probability we will use the following Borel theorem from Belyaev and Piterbarg (1972).
\begin{theorem}\label{Borel}
Let $\{X(t),t\in T\}$ be a real separable Gaussian process indexed by an arbitrary parameter set $T$, let
\begin{align*}
\sigma^2=\sup_{t\in T}\Var X(t)<\infty, \quad m=\sup_{t\in T}\mathbb{E}X(t)<\infty,
\end{align*}
and let the real number $b$ be such that
\begin{align*}
\mathbb{P}\Big(\sup_{t\in T}X(t)-\mathbb{E}X(t)\geq b\Big)\leq\frac{1}{2}.
\end{align*}
Then for all $x$
\begin{align*}
\mathbb{P}\Big(\sup_{t\in T}X(t)>x\Big)\leq2\bar\Phi\Big(\frac{x-m-b}{\sigma}\Big).
\end{align*}
\end{theorem}
%
%
There exists a constant $\zeta_1>0$ such that
\begin{align*}
\inf_{(i,j)\in V, t\in S_i^h,s\in S_j^h, 0<h\leq 1}\|t-s\|>\zeta_1,
\end{align*}
i.e., the distance between any two nonadjacent elements of the partition exceeds $\zeta_1$ uniformly in $h\in (0,1]$. This is due to the fact that the curvatures of the manifolds ${\cal M}_h$ is (uniformly) bounded, and that $V_r(S_j^h)$ is bounded away from zero uniformly in $j$ and $h$. See Lemma 3 of Genovese et al. (2012) for more details underlying this argument. The latter also implies that we can find a number $N_0>0$ such that $N_h$, the number of sets $S_i$, satisfies $N_h<N_0$ for all $h$. Assumption (\ref{SupGauss1}) implies that
\begin{align*}
\rho:=\sup_{\|t-s\|\geq\zeta_1, 0<h\leq 1}r_h(t,s)<1.
\end{align*}
We want to apply the above Borel theorem to $X_h(t)+X_h(s)$ with $t\in S_i^h$ and $s\in S_j^h$ and $(i,j)\in V$. To this end observe that
\begin{align*}
\sup_{0<h\leq 1}\sup_{t\in S_i^h,s\in S_j^h} \Var\left(X_h(t)+X_h(s)\right)\leq 2+2\rho
\end{align*}
and
\begin{align*}
\sup_{0<h\leq 1}\sup_{t\in S_i^h,s\in S_j^h} \mathbb{E}\left(X_h(t)+X_h(s)\right)=0.
\end{align*}
Next we show that there is a constant $b$ such that $\mathbb{P}\Big(\sup_{t\in S_i^h,s\in S_j^h} \left(X_h(t)+X_h(s)\right)>b\Big)\leq\frac{1}{2}$ 
for $h$ sufficiently small. Note that
\begin{align*}
\mathbb{P}\Big(\sup_{t\in S_i^h,s\in S_j^h} \left(X_h(t)+X_h(s)\right)>b\Big)&\leq \mathbb{P}\Big(\sup_{t\in J_{k,m_h},s\in J_{k,m_h}} \left(X_h(t)+X_h(s)\right)>b\Big)\\
&\leq \mathbb{P}\Big(\sup_{t\in J_{k,m_h}} X_h(t)>b/2\Big).
\end{align*}
All the arguments in {\bf Part 1} hold uniformly in $h$ as long as $\theta$ is large enough. In other words, the conclusions there can be restated by replacing $\theta$ with $x$ where $x\rightarrow\infty$. For instance, for any $\epsilon > 0$ we can choose $\max_{1\leq i \leq N_h}V_r(S_i^h)$  small enough such that
\begin{align*}
\mathbb{P}\Big(\sup_{t\in J_{k,m_h}} X_h(t)>x\Big)&\leq\sum_{i=1}^{N_h}\mathbb{P}\Big(\sup_{t\in S_i^h} X_h(t)>x\Big)\\
&\hspace*{-1cm}\leq(1+\epsilon)\,x^{2r/\alpha}\Psi(x)H_\alpha^{(r)} \int_{J_{k,m_h}}\|D_s^hM_s^h\|_rds
\end{align*}
holds for all $1 \leq k \leq m_h$ and $x > x_0$. Hence, since $x^{2r/\alpha}\Psi(x) \to 0$ as $x \to \infty$, we can find $b$ such that $\mathbb{P}(\sup_{t\in J_{k,m_h}} X_h(t)>b/2)<1/2$ for all $1 \leq k \leq m_h$ when $\max_{1\leq i \leq N_h}V_r(S_i^h)$  is sufficiently small. The above Borel inequality now gives (for large enough $\theta$) that
\begin{align}\label{PSupBound}
\mathbb{P}\Big(\sup_{t\in S_i^h,s \in S_j^h} \left(X_h(t)+X_h(s)\right)>2\theta\Big)&\leq 2\bar\Phi\bigg(\frac{\theta-b/2}{\sqrt{(1+\rho)/2}}\bigg).
\end{align}
%
Since the total number of elements in the sum in (\ref{DoubleSum}) is bounded by $N_h^2$, it follows from (\ref{IntersectionBound}) and (\ref{PSupBound}) that uniformly in $k$ (recall that the $B_i$ depend on $k$)
\begin{align}\label{SubSumBound1}
\sum\limits_{\substack{1\leq i<j\leq N_h, \\ j-i>1}}\mathbb{P}(B_i\cap B_j)\leq 2N_h^2\bar\Phi\bigg(\frac{\theta-b/2}{\sqrt{(1+\rho)/2}}\bigg)\leq 2N_0^2\bar\Phi\bigg(\frac{\theta-b/2}{\sqrt{(1+\rho)/2}}\bigg)=o(\theta^{2r/\alpha}\Psi(\theta)).
\end{align}
as $\theta \rightarrow \infty$ by using the well-known fact that $\lim_{u\to\infty}\frac{\bar\Phi(u)}{\Psi(u)}=1$ (see Cram$\acute{\rm e}$r, 1951, page 374).\\

Considering (\ref{SumB}), (\ref{DoubleSum}), (\ref{SubSumBound2}) and (\ref{SubSumBound1}) and their respective conditions, we have 
\begin{align}\label{ProToSum}
\mathbb{P}\Big(\sup_{t\in J_{k,m_h}} X_h(t)>\theta\Big)=(1+ o(1))\,\sum_{i=1}^{N_h}\mathbb{P}\Big(\sup_{t\in S_i^h} X_h(t)>\theta\Big)\quad \text{as }\,\theta \to \infty,
\end{align}
where the $o(1)$-term is uniform in $k$.\\

Combining (\ref{SumToInt}) and (\ref{ProToSum}), we have for $\sup_{h \in (0,1]}\max_{1\leq i \leq N_h}V_r(S_i^h)$ sufficiently small, that
\begin{align}\label{G7}
\sum_{k\leq m_h}\mathbb{P}\Big(\sup_{t\in J_{k,m_h}}X_h(t)> \theta\Big)=(1+ o(1))\theta^{2r/\alpha}\Psi(\theta)H_\alpha^{(r)} \int_{\mathcal {M}_h}\|D_s^hM_s^h\|_rds\quad\text{as\;}\,h \to 0.
\end{align}

{\bf Part 3.}
Note that from the expression of $\theta$ in (\ref{ThetaExp}) we have  for any fixed $z$ 
\begin{align}\label{PhiX}
\theta^{2r/\alpha}\Psi(\theta)=\frac{\theta^{2r/\alpha-1}}{\sqrt{2\pi}}\exp\Big\{-\frac{\theta^2}{2}\Big\}=\frac{h^r\exp\{-z\}}{H_\alpha^{(r)} \int_{\mathcal {M}_1}\|D_s^0M_s^1\|_rds}(1+o(1))=O(h^r)
\end{align}
as $h\to0$.\\

Observing that  $\max_{1\leq k \leq m_h}V_r(J_{k,m_h}^{-\delta}) = O(\delta)$ (uniformly in $h$), and using (\ref{PhiX}) we obtain for $h$ small enough that
%
{\allowdisplaybreaks
\begin{align*}
0&\leq \mathbb{P}\Big(\sup_{t\in \mathcal {M}_h}X_h(t)> \theta\Big)-\mathbb{P}\Big(\sup_{t\in \cup_{k\leq m_h}J_{k,m_h}^\delta}X_h(t) > \theta\Big)\\
&\leq \mathbb{P}\Big(\sup_{t\in \mathcal {M}_h\backslash\cup_{k\leq m_h}J_{k,m_h}^\delta}X_h(t) > \theta\Big)\\
&\leq \sum_{k=1}^{m_h}\mathbb{P}\Big(\sup_{t\in J_{k,m_h}^{-\delta}}X_h(t) > \theta\Big)\\
&\leq (1+\epsilon)\,\theta^{2r/\alpha}\Psi(\theta)H_\alpha^{(r)} \sum_{k=1}^{m_h}\int_{J_{k,m_h}^{-\delta}}\|D_s^hM_s^h\|_rds\\
&\leq O(\delta) (1+\epsilon) H_\alpha^{(r)} c m_h\theta^{2r/\alpha}\Psi(\theta)\\
&\leq O(\delta) (1+\epsilon)\,H_\alpha^{(r)} c\; O((hl^*)^{-r}) \frac{h^r\exp\{-z\}}{H_\alpha^{(r)} \int_{\mathcal {M}_1}\|D_s^0M_s^1\|_rds}\\
&= O(\delta),
\end{align*}
}
uniformly in $k.$ Here $c$ is from (\ref{DMbound}).  
Similarly, (and again uniformly in $k$) we have 
%
$ 0\leq \mathbb{P}\big(\inf_{t\in \mathcal {M}_h}X_h(t) < -\theta\big)-\mathbb{P}\big(\inf_{t\in \cup_{k\leq m_h}J_{k,m_h}^\delta}X_h(t) < -\theta\big) = O(\delta)$ uniformly in $0 < h \le h_1$ for some $h_1 > 0$. 
%
Collecting what we have we get that uniformly in $0 < h \le h_1$ 
%
%
\begin{align}\label{G1}
\mathbb{P}\Big(\sup_{t\in \mathcal {M}_h}|X_h(t)|\leq \theta\Big)=\mathbb{P}\Big(\sup_{t\in \cup_{k\leq m_h}J_{k,m_h}^\delta}|X_h(t)| \leq \theta\Big)+ O(\delta)
\end{align}
and
\begin{align}\label{G1T}
\sum_{k=1}^{m_h}\mathbb{P}\Big(\sup_{t\in J_{k,m_h}}|X_h(t)| > \theta\Big)=\sum_{k=1}^{m_h}\mathbb{P}\Big(\sup_{t\in J_{k,m_h}^\delta}|X_h(t)| > \theta\Big)+O(\delta).
\end{align}
%

{\bf Part 4.} Here we show that replacing $\cup_{k \le m_h} J^\delta_{k,m_h}$ by the dense `grid' ${\mathbb T}^\delta_h$ (see (v)) leads to a negligible error in the corresponding extreme value probabilities.\\
 
We write ${\mathbb T}^\delta_h = \bigcup_{1\leq k \leq m_h} \Gamma_{\gamma \theta^{-2/\alpha}}(J_{k,m}^\delta)$ as $\{t_j,j=1,\cdots, N_h^*\}.$  Our assumptions assure that $N_h^* =O(\frac{\theta^{2r/\alpha}}{h^r\gamma^r})$, because $V_r({\cal M}_h) = O(h^{-r})$ and the `mesh size' of the curvilinear mesh on ${\cal M}_h$ is $O( \frac{\theta^{2/\alpha}}{\gamma^r})$, due to the construction of the triangulation and the uniformly bounded curvature on the manifolds ${\cal M}_h$.\\

With (\ref{PieceSum}), (\ref{DMbound}) and (\ref{PhiX}), we have
\begin{align}\label{XoverPiece}
\mathbb{P}\Big(\max_{t_j \in J_{k,m_h}^\delta}|X_h(t_j)|> \theta\Big)=O(h)
\end{align}
uniformly in $k$ as $h\to0$. Note that here and below we for brevity omit to indicate that the maxima (or minima, respectively) run over $j = 1,\ldots, N_n^*$ (i.e. over all $t_j \in {\mathbb T}^\delta_h$). It follows that as $h \to 0$
\begin{align}\label{G4}
\sum_{k=1}^{m_h}\log{\Big(1-\mathbb{P}\Big(\max_{t_j\in J_{k,m_h}^\delta}|X_h(t_j)|> \theta\Big)\Big)}=(1+o(1))\sum_{k=1}^{m_h}\mathbb{P}\Big(\max_{t_j\in J_{k,m_h}^\delta}|X_h(t_j)|> \theta\Big).
\end{align}
%
It follows from (\ref{PieceSum}) and its version with the $\max$ over the discrete set replaced by the $\sup$ over $t \in S_i^h$ (see discussion given below (\ref{PieceSum2})), that for any $\epsilon > 0$ there exists thresholds for $h$, $\gamma$ and the norm of partitions, such that 
\begin{align*}
0&\leq \mathbb{P}\Big(\sup_{t\in J_{k,m_h}^\delta}X_h(t)>\theta\Big)-\mathbb{P}\Big(\max_{t_j\in J_{k,m_h}^\delta}X_h(t_j)>\theta\Big)\\
&\leq \sum_{i=1}^{N_h}\Big[\mathbb{P}\Big(\sup_{t\in S_i^h}X_h(t)>\theta\Big)-\mathbb{P}\Big(\max_{t_j\in S_i^h}X_h(t_i)>\theta\Big)\Big]\\
&\leq\epsilon\;\theta^{2r/\alpha}\Psi(\theta)H_\alpha^{(r)} \int_{J_{k,m_h}^\delta}\|D_s^hM_s^h\|_rds,
\end{align*}
provided $h$, $\gamma$ and the norm of partitions are smaller then their respective thresholds. Similarly, (\ref{PieceSum2}) and its corresponding `continuous' version imply that for $h$ and $\gamma$ smaller than their respective thresholds indicated in {\bf Part 1}, we have
\begin{align*}
0\leq \mathbb{P}\Big(\inf_{t\in J_{k,m_h}^\delta}X_h(t)<-\theta\Big)-\mathbb{P}\Big(\min_{t_i\in J_{k,m_h}^\delta}X_h(t_i)<-\theta\Big)\leq\epsilon\,\theta^{2r/\alpha}\Psi(\theta)H_\alpha^{(r)} \int_{J_{k,m_h}^\delta}\|D_s^hM_s^h\|_rds.
\end{align*}
Consequently, if $h$ and $\gamma$ and $\max_{1\leq k \leq m_h}V_r(J^\delta_{k,m_h})$ are small enough, we have
{\allowdisplaybreaks
\begin{align}\label{SupDiff}
0&\leq \mathbb{P}\Big(\sup_{t\in \bigcup_{k\leq m_h}J_{k,m_h}^\delta}|X_h(t)|> \theta\Big)-\mathbb{P}\Big(\max_{t_j\in \bigcup_{k\leq m_h}J_{k,m_h}^\delta}|X_h(t_j)|> \theta\Big)\nonumber\\
&\leq\sum_{k=1}^{m_h}\Big[\mathbb{P}\Big(\sup_{t\in J_{k,m_h}^\delta}|X_h(t)|> \theta\Big)-\mathbb{P}\Big(\max_{t_j\in J_{k,m_h}^\delta}|X_h(t_j)|> \theta\Big)\Big]\nonumber\\
&\leq\sum_{k=1}^{m_h}\Big[\mathbb{P}\Big(\sup_{t\in J_{k,m_h}^\delta}X_h(t)> \theta\Big)+\mathbb{P}\Big(\inf_{t\in J_{k,m_h}^\delta}X_h(t)<- \theta\Big)-\mathbb{P}\Big(\max_{t_j\in J_{k,m_h}^\delta}X_h(t_j)> \theta\Big)\nonumber\\
&\hspace{1cm}-\mathbb{P}\Big(\min_{t_j\in J_{k,m_h}^\delta}X_h(t_j)<- \theta\Big)\Big]\nonumber\\
&\leq2\epsilon\,\theta^{2r/\alpha}\Psi(\theta)H_\alpha^{(r)} \int_{\bigcup_{k\leq m_h}J_{k,m_h}^\delta}\|D_s^hM_s^h\|_rds\nonumber\\
&\leq2\epsilon\,\theta^{2r/\alpha}\Psi(\theta)H_\alpha^{(r)} \int_{\mathcal{M}_h}\|D_s^hM_s^h\|_rds.
\end{align}
}
To see the order of the upper bound in (\ref{SupDiff}), by the dominated convergence theorem (and using our assumption on the behavior of $D^h_s$) we have
\begin{align}\label{Integration}
\frac{h^r\int_{\mathcal {M}_h}\|D_s^hM_s^h\|_rds}{\int_{\mathcal {M}_1}\|D_s^0M_s^1\|_rds} = \frac{\int_{\mathcal {M}_1}\|D_{s/h}^hM_s^1\|_rds}{\int_{\mathcal {M}_1}\|D_s^0M_s^1\|_rds}\to 1, \quad \textrm{as} \quad h\rightarrow 0.
\end{align}

As a result of (\ref{PhiX}) and (\ref{Integration}), we can write for $\max_{1\leq k \leq m_h}V_r(J^\delta_{k,m_h})$ small enough that
\begin{align}\label{G2}
\mathbb{P}\Big(\sup_{t\in \bigcup_{k\leq m_h}J_{k,m_h}^\delta}|X_h(t)|\leq \theta\Big)=\mathbb{P}\Big(\max_{t_j\in \bigcup_{k\leq m_h}J_{k,m_h}^\delta}|X_h(t_j)|\leq \theta\Big)+ o(1)
\end{align}
and
\begin{align}\label{G2T}
\sum_{k=1}^{m_h}\mathbb{P}\Big(\sup_{t\in J_{k,m_h}^\delta}|X_h(t)|> \theta\Big)=\sum_{k=1}^{m_h}\mathbb{P}\Big(\max_{t_j\in J_{k,m_h}^\delta}|X_h(t_j)|> \theta\Big)+ o(1),
\end{align}
as $\gamma,h\to0.$\\

{\bf Part 5.} Here we find an upper bound for the difference
\begin{align}\label{PDiff}
\bigg|\mathbb{P}\Big(\max_{t_j\in\bigcup_{k\leq m_h}J_{k,m_h}^\delta}|X_h(t_j)|\leq \theta\Big)-\prod_{k\leq m_h}\mathbb{P}\Big(\max_{t_j\in J_{k,m_h}^\delta}|X_h(t_j)|\leq \theta\Big)\bigg|.
\end{align}
This step uses similar ideas as in the proof of Lemma 5.1 in Berman (1971).  \\

Define a probability measure $\tilde{\mathbb{P}}$ such that for any $x_{t_j}\in \mathbb{R}$ with $t_j\in\bigcup_{k\leq m_h}J_{k,m_h}^\delta$,
\begin{align*}
\tilde{\mathbb{P}}\Big(X_h(t_j)\leq x_{t_j}, t_j\in\bigcup_{k\leq m_h}J_{k,m_h}^\delta\Big)=\prod_{k\leq m_h}\mathbb{P}\Big(X_h(t_j)\leq x_{t_j}, t_j\in J_{k,m_h}^\delta\Big),
\end{align*}
i.e., under $\tilde{\mathbb{P}}$ the vectors $(X_h(t_i):t_i\in J_{k,m}^\delta \cap {\mathbb T}^\delta_h)$ and $(X_h(t_j):t_j\in J_{k^\prime,m}^\delta \cap {\mathbb T}^\delta_h)$ independent for $k\neq k^\prime$. 
%
By Lemma \ref{BickelLemma}, the difference in (\ref{PDiff}) can be bounded by
\begin{align}\label{TripleSum}
&8\sum\limits_{\substack{k\leq m_h,k^{\prime}\leq m_h,\\ k\neq k^\prime}}\sum_{t_i\in J_{k,m_h}^\delta}\sum_{t_j\in J_{k^\prime,m_h}^\delta}\int_0^{|r_h(t_i,t_j)|}\phi(\theta,\theta,\lambda)d\lambda \nonumber\\
&=8\sum\limits_{\substack{k\leq m_h,k^{\prime}\leq m_h,\\ k\neq k^\prime}}\sum_{t_i\in J_{k,m_h}^\delta}\sum_{t_j\in J_{k^\prime,m_h}^\delta}\int_0^{|r_h(t_i,t_j)|}\frac{1}{2\pi(1-\lambda^2)^{1/2}}\exp{\bigg(-\frac{\theta^2}{1+\lambda}\bigg)}d\lambda\nonumber\\
&\leq8\sum\limits_{\substack{k\leq m_h,k^{\prime}\leq m_h,\\ k\neq k^\prime}}\sum_{t_i\in J_{k,m_h}^\delta}\sum_{t_j\in J_{k^\prime,m_h}^\delta}\frac{|r_h(t_i,t_j)|}{2\pi(1-(r_h(t_i,t_j))^2)^{1/2}}\exp{\bigg(-\frac{\theta^2}{1+|r_h(t_i,t_j)|}\bigg)}.
\end{align}
(Note that here the notation $\{t_i\in J_{k,m_h}^\delta\}$ is a shortcut for $\{ t_i\in J_{k,m_h}^\delta \cap {\mathbb T}^\delta_h\}$, and similarly for $t_j$.) For $t_i\in J_{k,m_h}^\delta$ and $t_j\in J_{k^\prime,m_h}^\delta$ with $k\neq k^\prime$, it follows from the uniform  boundedness of the curvature of the growing manifold that there exists a positive real $\varsigma$ such that $\|t_i-t_j\|\geq \varsigma$, uniformly for all $0<h\leq 1$. (Similar arguments have been used above already.) Thus we obtain from assumption (\ref{SupGauss1}) that there exists $\eta > 0$  dependent on $\varsigma$ such that
\begin{align}\label{RijBound}
|r_h(t_i,t_j)|<\eta < 1
\end{align}
uniformly in $t_i\in J_{k,m_h}^\delta$ and $t_j\in J_{k^\prime,m_h}^\delta$ with $k\neq k^\prime$ and $0<h\leq 1$.\\

Let $\omega$ be an arbitrary number satisfying
\begin{align*}
0<\omega<\frac{2}{(1+\eta)}-1.
\end{align*}
We take $\gamma = v(h^{-1})^{1/3r}$ in what follows and divide the triple sum in (\ref{TripleSum}) into two parts: In one part the indices $i,j$ are constrained such that $\|t_i-t_j\|<(N_h^*)^{\omega/r}\gamma \theta^{-2/\alpha}$ and for the other part the indices take the remaining values. In the first part, the number of summands in the triple sum is of the order $O((N_h^*)^{\omega+1})$, because there is a total of $O(N^*_h)$ points and for each of this points we have to consider at most  $O((N_h^*)^{\omega})$ pairs. Taking (\ref{RijBound}) into account, we get the order of the sum in the first part of (\ref{TripleSum})
\begin{align*}
O\bigg((N_h^*)^{\omega+1}\exp\bigg\{-\frac{\theta^2}{1+\eta}\bigg\}\bigg)&=O\bigg(\bigg(\frac{\theta^{2r/\alpha}}{h^r\gamma^r}\bigg)^{1+\omega}\exp\bigg\{-\frac{\theta^2}{1+\eta}\bigg\}\bigg)\\
&=O\bigg(\bigg(\frac{(\log{h^{-1}})^{r/\alpha}}{h^r\gamma^r}\bigg)^{1+\omega}\exp\bigg\{-\frac{2r\log{h^{-1}}}{1+\eta}\bigg\}\bigg)\\
&=O\bigg(h^{\frac{2r}{1+\eta}-r(1+\omega)}\Big(\log{h^{-1}}\Big)^{\frac{(1+\omega) r}{\alpha}}\Big(v(h^{-1})\Big)^{-\frac{(1+\omega)\gamma}{3r}}\bigg),
\end{align*}
which tends to zero as $h$ approaches zero.\\

Then we consider the second part of (\ref{TripleSum}) with $\|t_i-t_j\|\geq (N_h^*)^{\omega/r}\gamma \theta^{-2/\alpha}$. Noticing $(1+|r^h(t_i,t_j)|)^{-1}\geq 1-|r^h(t_i,t_j)|$ and (\ref{RijBound}), we can have the following bound for the second part of (\ref{TripleSum}):
\begin{align}\label{Bound2sum1}
8\exp(-\theta^2)\sum\limits_{\substack{k\leq m_h,k^{\prime}\leq m_h,\\ k\neq k^\prime}}\sum\limits_{\substack{t_i\in J_{k,m_h}^\delta, t_j\in J_{k^\prime,m_h}^\delta, \\ \|t_i-t_j\|\geq(N_h^*)^{\omega/r}\gamma \theta^{-2/\alpha}}}\frac{|r^h(t_i,t_j)|}{2\pi(1-\eta^2)^{1/2}}\exp{(\theta^2|r^h(t_i,t_j)|)}.
\end{align}

By (\ref{SupGauss2}) and the fact that $\theta^2=O(\log{h^{-1}})$, we have that $\sup_{\|t_i-t_j\|\geq(N_h^*)^{\omega/r}\gamma \theta^{-2/\alpha}} \theta^2|r^h(t_i,t_j)| \rightarrow 0$ as $h\rightarrow 0$. Hence (\ref{Bound2sum1}) is of the order of
\begin{align}\label{Bound2sum2}
h^{2r}\sum\limits_{\substack{k\leq m_h,k^{\prime}\leq m_h,\\ k\neq k^\prime}}\sum\limits_{\substack{t_i\in J_{k,m_h}^\delta, t_j\in J_{k^\prime,m_h}^\delta, \\ \|t_i-t_j\|\geq(N_h^*)^{\omega/r}\gamma \theta^{-2/\alpha}}}|r^h(t_i,t_j)|
\end{align}
When $h$ is sufficiently small we have
\begin{align}
\sup_{\|t_i-t_j\|\geq(N_h^*)^{\omega/r}\gamma \theta^{-2/\alpha}}|r^h(t_i,t_j)|\leq \frac{v((N_h^*)^{\omega/r}\gamma \theta^{-2/\alpha})}{[\log((N_h^*)^{\omega/r}\gamma \theta^{-2/\alpha})]^{2r/\alpha}}.
\end{align}
Therefore, due to (\ref{SupGauss2}), (\ref{Bound2sum2}) is of the order
\begin{align*}
&O\left(h^{2r}(N_h^*)^2\frac{v((N_h^*)^{\omega/r}\gamma \theta^{-2/\alpha})}{[\log((N_h^*)^{\omega/r}\gamma \theta^{-2/\alpha})]^{2/\alpha}}\right)\\
&=O\left(\frac{[\log(h^{-1})]^{2r/\alpha}v((N_h^*)^{\omega/r}\gamma \theta^{-2/\alpha})}{\bigg[\log \bigg(h^{-\omega} \Big([\log (h^{-1})]^{1/\alpha}v(h^{-1})^{-1/3r}\Big)^{\omega-1}\bigg)\bigg]^{2r/\alpha}v(h^{-1})^{2/3}}\right)\\
&=o(1) \quad \textrm{as} \quad h\rightarrow 0.
\end{align*}
Now we have proved (\ref{TripleSum}) tends to zero as $h$ goes to zero.  So we have with this choice of $\gamma$ that as $h\to0$ 
\begin{align}\label{G3}
\mathbb{P}\Big(\max_{t_j\in\bigcup_{k\leq m_h}J_{k,m_h}^\delta}|X_h(t_j)|\leq \theta\Big)=\prod_{k\leq m_h}\mathbb{P}\Big(\max_{t_j\in J_{k,m_h}^\delta}|X_h(t_j)|\leq \theta\Big)+o(1),
\end{align}
where $\delta>0$ is fixed small enough.\\

{\bf Final step:} Now we collect all the approximations above, including (\ref{G1}), (\ref{G1T}), (\ref{G2}), (\ref{G2T}), (\ref{G3}), (\ref{G4}), (\ref{G7}), (\ref{PhiX}) and (\ref{Integration}). We have for $\delta > 0$ and $\sup_{h \in (0,1]}\max_{1\leq i \leq N_h}V_r(S_i^h)$ fixed and chosen small enough, and $\gamma = v(h^{-1})^{1/3r}$ that as $h \to 0$
{\allowdisplaybreaks
\begin{align*}
&\mathbb{P}\Big(\sup_{t\in\mathcal {M}_h}|X_h(t)|\leq \theta\Big)\\
&\overset{(\ref{G1})}{=\!=}\mathbb{P}\Big(\sup_{t\in\bigcup_{k\leq m_h}J_{k,m_h}^\delta}|X_h(t)|\leq \theta\Big)+ o(1)\\
&\overset{(\ref{G2})}{=\!=}\mathbb{P}\Big(\max_{t_j\in\bigcup_{k\leq m_h}J_{k,m_h}^\delta}|X_h(t_j)|\leq \theta\Big)+o(1)\\
&=\!=\mathbb{P}\bigg(\bigcap_{k\leq m_h}\Big(\max_{t_j\in J_{k,m_h}^\delta}|X_h(t_j)|\leq \theta\Big)\bigg)+o(1)\\
&\overset{(\ref{G3})}{=\!=}\prod_{k\leq m_h}\mathbb{P}\Big(\max_{t_j\in J_{k,m_h}^\delta}|X_h(t_j)|\leq \theta\Big)+o(1)\\
&=\!=\exp\bigg\{\sum_{k\leq m_h}\log{\Big(1-\mathbb{P}\Big(\max_{t_j\in J_{k,m_h}^\delta}|X_h(t_j)|> \theta\Big)\Big)}\bigg\}+o(1)\\
&\overset{(\ref{G4})}{=\!=}\exp\bigg\{-(1+ o(1))\sum_{k\leq m_h}\mathbb{P}\Big(\max_{t_j\in J_{k,m_h}^\delta}|X_h(t_j)|> \theta\Big)\bigg\}+ o(1)\\
&\overset{(\ref{G2T})}{=\!=}\exp\bigg\{-(1+o(1))\bigg[\sum_{k\leq m_h}\mathbb{P}\Big(\sup_{t\in J_{k,m_h}^\delta}|X_h(t)|> \theta\Big)- o(1)\bigg]\bigg\}+o(1)\\
&\overset{(\ref{G1T})}{=\!=}\exp\bigg\{-2(1+o(1)) \sum_{k\leq m_h}\mathbb{P}\Big(\sup_{t\in J_{k,m_h}}X_h(t)> \theta\Big) \bigg\} + o(1)\\
%
%
&\overset{(\ref{G7})}{=\!=}\exp\bigg\{-2(1+o(1)) \;\theta^{2r/\alpha}\Psi(\theta)H_\alpha^{(r)} \int_{\mathcal {M}_h}\|D_s^hM_s^h\|_rds \bigg\}+o(1).
\end{align*}
}
This completes our proof by using (\ref{PhiX}), (\ref{Integration}).

\end{section}

\begin{section}{Miscellaneous}\label{misc}

In this section we collect some miscellaneous results and definitions that are needed in the above proof. We present them in a separate section in order to not interrupt the flow of the above proof.\\

{\em Definition of generalized Pickands constant} (following Piterbarg and Stamatovich, 2001).  For $0<\alpha\leq 2$, let $\chi_\alpha(t)$ be a continuous Gaussian field with $\mathbb{E}\chi_\alpha(t)=-\|t\|^\alpha$ and $\Cov(\chi_\alpha(t), \chi_\alpha(s))=\|t\|^\alpha+\|s\|^\alpha-\|t-s\|^\alpha$ where $s,t\in \mathbb{R}^n$. The existence of such a field $\chi_\alpha(t)$ follows from Mikhaleva and Piterbarg (1997).\\

For any compact set $\mathcal{T}\subset\mathbb{R}^n$ define
\begin{align*}
H_\alpha(\mathcal{T})=\mathbb{E}\exp\Big(\sup_{t\in \mathcal{T}}\chi_\alpha(t)\Big).
\end{align*}

Let $D$ be a non-degenerated $n \times n$ matrix. For a set $A \subset \R^n$ let $DA = \{Dx, x\in A\}$ denote the image of $A$ under $D$. For any $q>0$, we let
\begin{align*}
[0,q]^r = \{t: t_i \in [0,q], i= 1,\cdots, r;\, t_i=0, i=r+1, \cdots, n\},
\end{align*}
denote a cube of dimension $r$ generated by the first $r$ coordinates in $\mathbb{R}^n$. Let
\begin{align*}
H_\alpha^{D\mathbb{R}^r} := \lim_{q\rightarrow\infty}\frac{H_\alpha(D[0,q]^r)}{\lambda_r(D[0,q]^r)},
\end{align*}
where $\lambda_r$ denotes Lebesgue measure in $\R^r$. It is known that $H_\alpha^{D\mathbb{R}^r}$ exists and $0< H_\alpha^{D\mathbb{R}^r} <\infty$ (see Belyaev
and Piterbarg, 1972). With $D = I$ the unit matrix, we write $H_\alpha^{(r)}=H_\alpha^{I\mathbb{R}^r}$. Since by definition the random field $\chi_\alpha(\cdot)$ is isotropic, $H_\alpha^{D\mathbb{R}^r} = H_\alpha^{(r)}$ for any orthogonal matrix $D$. The constant $H_\alpha:=H_\alpha^{(n)}$ is the (generalized) {\em Pickands constant}.\\

Further, for positive integers $l$ and $\gamma>0$, let
\begin{align*}
C^r(l,\gamma)&=\{t\gamma: t_i\in [0,l] \cap \mathbb{N}_0, i=1,\cdots, r; t_i=0, i=r+1, \cdots, n\} \\[4pt]
&= \gamma\;\big([0,l]^r\; \cap {\mathbb N}_0^n\big),
\end{align*}
let $H_\alpha^{D,(r)}(l,\gamma)=H_\alpha (DC^r(l,\gamma))$. Again, for $D$ orthogonal and due to isotropy of $\chi_\alpha(\cdot)$, we just write $H_\alpha^{(r)}(l,\gamma)=H_\alpha^{D,(r)}(l,\gamma)$. We let 
$$H_\alpha^{(r)}(\gamma)= {\displaystyle \lim_{l\rightarrow\infty}}\frac{H_\alpha^{(r)}(l,\gamma)}{l^r}$$ 
assuming this limit exists, and  for $r = n$ we simply write $H_\alpha(l,\gamma)$ and $H_\alpha(\gamma) $ instead of $H_\alpha^{(n)}(l,\gamma)$ and $H_\alpha^{(n)}(\gamma)$, respectively. 
%
%
We have the following lemma from Bickel and Rosenblatt (1973b).
\begin{lemma}\label{HGamma}
$H_\alpha^{(r)}=\lim_{\gamma\rightarrow0}\frac{H_\alpha^{(r)}(\gamma)}{\gamma^r}$.
\end{lemma}

In the following we present further results for Gaussian fields that are used in the proofs.

\begin{lemma}\label{Slepian} (\textbf{Slepian's lemma; see Slepian, 1962})
Let $\{X_t, \,t \in T\}$ and $\{Y_t, \,t \in T\}$ be Gaussian processes satisfying the assumptions of Theorem \ref{Borel} with the same mean functions. If the covariance functions $r_X(s,t)$ and $r_Y (s, t)$ meet the relations
\begin{align*}
r_X(t, t) \equiv r_Y (t, t),\hspace{0.5cm} t \in T \hspace{2cm} r_X(s, t) \leq r_Y (s,t), \hspace{0.5cm}t, s \in T,
\end{align*}
%
%
%
then for any $x$
\begin{align*}
\mathbb{P}\Big\{\sup_{t\in T}X_t<x\bigg\}\leq \mathbb{P}\bigg\{\sup_{t\in T}Y_t<x\Big\}.\\
\end{align*}
\end{lemma}

We also need this result from Piterbarg (1996). Recall that $\Psi$ is defined at the beginning of section~\ref{proof}. 

\begin{lemma}\label{Zbigniew}(\textbf{Lemma 6.1 of Piterbarg, 1996})
Let $X(t)$ be a continuous homogeneous Gaussian field where $t\in \mathbb{R}^n$ with expected value $\mathbb{E}X(t)=0$ and covariance function $r(t)$ satisfying
\begin{align*}
r(t)=\mathbb{E}(X(t+s)X(s))=1-\|t\|^\alpha+o(\|t\|^\alpha).
\end{align*}
Then for any compact set $\mathcal{T}\subset\mathbb{R}^n$ 
\begin{align*}
\mathbb{P}\Big(\sup_{t\in u^{-2/\alpha}\mathcal{T}}X(t)>u\Big)=\Psi(u)H_\alpha(\mathcal{T})(1+o(1)) \quad \text{as }\,u\rightarrow \infty.\\
\end{align*}
\end{lemma}
The next result follows immediately.\\

\begin{corollary}\label{BickelA1}
Let $X(t)$ be as in Lemma~\ref{Zbigniew}. Let $M_k\in\mathbb{R}^n, k=1,\cdots,n$ be a basis of $\mathbb{R}^n$, $l\in \mathbb{Z}^+$ and $\gamma>0$. We have  with $C^r(l,1)$ as defined on page 2 that
\begin{align*}
\lim_{x\rightarrow\infty}\frac{\mathbb{P}( \max_{(i_1,\ldots,i_n) \in C^r(l,1)}X(\sum_{k=1}^ni_k\gamma x^{-2/\alpha}M_k)>x)}{\Psi(x)}=H_\alpha^{(r)}(l,\gamma).
\end{align*}
\begin{align*}
\lim_{x\rightarrow\infty}\frac{\mathbb{P}(\min_{(i_1,\ldots,i_n) \in C^r(l,1)}X(\sum_{k=1}^ni_k\gamma x^{-2/\alpha}M)<-x)}{\Psi(x)}=H_\alpha^{(r)}(l,\gamma).
\end{align*}
\end{corollary}

\begin{remark}
This is also a simple extension of Lemma A1 of Bickel and Rosenblatt (1973a).
\end{remark}

\begin{lemma}\label{Pickands}(\textbf{Lemma 2.3 of Pickands, 1969})
Let $X$ and $Y$ be jointly normal, mean zero with variances 1 and covariance $r$. Then
\begin{align*}
\mathbb{P}(X>x,Y>x)\leq(1+r)\Psi(x)\bigg(1-\Phi\bigg(x\sqrt{\frac{1-r}{1+r}\;}\;\bigg)\;\bigg).
\end{align*}
\end{lemma}

The next lemma is an extension of Lemma A3 in Bickel and Rosenblatt (1973a), Lemma 3 and and Lemma 5 of Bickel and Rosenblatt (1973b) and Lemma 2.5 in Pickands (1969). Its proof is also adapted from the three sources.
\begin{lemma}\label{Piece}
Let $X(t)$ be a centered homogeneous Gaussian field on $\mathbb{R}^n$ with covariance function
\begin{align*}
r(t)=\mathbb{E}(X(t+s)X(s))=1-\|t\|^\alpha+o(\|t\|^\alpha).
\end{align*}
Let $\mathcal{T}$ be a Jordan measurable set imbedded in a $r$-dimensional linear space with $V_r({\cal T}) = \lambda<\infty$. For $\gamma,x > 0$ let $\mathcal{G}(\mathcal{T},\gamma,x)$ be a collection of points defining a mesh contained in $\mathcal{T}$ with mesh size $\gamma x^{-2/\alpha}.$ Assume
\begin{align}\label{InfR}
\xi(\|t\|):=\inf_{0< \|s\| \leq \|t\|}\|s\|^{-\alpha}(1-r(s))/2>0\quad\text{for $\|t\|$ small enough.}
\end{align}
Then
\begin{align}\label{BickelExtension1}
\lim_{x\rightarrow\infty}\frac{\mathbb{P}(\max\{X(t): t\in \mathcal{G}(\mathcal{T},\gamma,x)\}>x)}{x^{2r/\alpha}\Psi(x)}=\lambda\frac{H_\alpha^{(r)}(\gamma)}{\gamma^r}
\end{align}
and
\begin{align}\label{BickelExtension2}
\lim_{x\rightarrow\infty}\frac{\mathbb{P}(\sup\{X(t):t\in\mathcal{T}\}>x)}{x^{2r/\alpha}\Psi(x)}=\lambda H_\alpha^{(r)}
\end{align}
uniformly in $\mathcal{T}\in\mathcal{E}_c$ where $\mathcal{E}_c$ is the collection of all $r$-dimensional Jordan measurable sets with $r$-dimensional Hausdorff measure bounded by $c<\infty$. Similarly,
\begin{align}\label{BickelExtension3}
\lim_{x\rightarrow\infty}\frac{\mathbb{P}(\inf\{X(t):t\in\mathcal{T}\}<-x)}{x^{2r/\alpha}\Psi(x)}=\lambda H_\alpha^{(r)}.
\end{align}
uniformly in $\mathcal{T}\in\mathcal{E}_c.$ 
\end{lemma}
\begin{proof}
The results in Lemma 3 and and Lemma 5 of Bickel and Rosenblatt (1973b) are similar but they are only given for two-dimensional squares. It is straightforward to generalized them to  hyperrectangles and further to Jordan measurable sets.
\end{proof}

\begin{theorem}\label{Piterbarg}(\textbf{Theorem 2 of Piterbarg and Stamatovich, 2001}) Let $\{X(t), t\in \mathbb{R}^n\}$ be a Gaussian centered locally $(\alpha,D_t)$-stationary field with a continuous matrix function $D_t$. Let $\mathcal {M}\subset \mathbb{R}^n$ be a smooth compact of dimension $r$. Then
\begin{align*}
\frac{\mathbb{P}(\sup_{t\in\mathcal {M}}X(t)>x)}{x^{2r/\alpha}\Psi(x)}\to H_\alpha^{(r)}\int_{\mathcal {M}}\|D_sM_s\|_rds
\end{align*}
as $x\rightarrow\infty$, where  $M_s$ is an $n\times r$ matrix with columns the orthonormal basis of the linear subspace tangent to $\mathcal {M}$ at $s$.\\
\end{theorem}

%

\begin{lemma} \label{BickelLemma} (\textbf{Lemma A4 of Bickel and Rosenblatt, 1973a})
Let
\begin{align*}
\phi(x,y,\rho)=\frac{1}{2\pi(1-\rho^2)^{1/2}}\exp\bigg\{-\frac{x^2-2\rho xy+y^2}{2(1-\rho^2)}\bigg\}.
\end{align*}
Let $\Sigma_1=\{r_{ij}\},\Sigma_2=\{s_{ij}\}$ be $N\times N$ nonnegative semi-definite matrices with $r_{ii}=s_{ii}=1$ for all $i$. Let $X=(X_1,\cdots,X_N)$ be a mean 0 Gaussian vector with covariance matrix $\Sigma_1$ under probability measure $\mathbb{P}_{\Sigma_1}$ or $\Sigma_2$ under $\mathbb{P}_{\Sigma_2}$. Let $u_1,\cdots,u_N$ be nonnegative numbers and $u=\min_ju_j$. Then
\begin{align*}
|\mathbb{P}_{\Sigma_1}[X_j\leq u_j,1\leq j\leq N]-\mathbb{P}_{\Sigma_2}[X_j\leq u_j,1\leq j\leq N]|\leq 4\sum_{i,j}\bigg|\int_{s_{ij}}^{r_{ij}}\phi(u,u,\lambda)d\lambda\bigg|.
\end{align*}
\end{lemma}

\newpage

{\bf References}

\begin{description}
\item {\sc Adler, R.J.} (2000). On excursion sets: tube formulas and maxima of random fields. {\em Ann. Appl. Prob.} {\bf 10}, 1--74.
\item {\sc Adler, R.J. and Taylor, J.E.} (2007). \emph{Random Fields and Geometry}, Springer, New York.
\item {\sc Aza\"{i}s, J.-M. and Wschebor, M.} (2009). \emph{Level Sets and Extrema of Random Processes and Fields}, John Wiley \& Sons, Hoboken, NJ.
\item {\sc Belyaev Yu.K. and Piterbarg, V.I.} (1972). The asymptotic behavior of the average number of the A-points of upcrossings of a Gaussian field beyond a high level. {\em Akad. Nauk SSSR} {\bf 203}, 9--12.
\item {\sc Berman, M.S.} (1971). Asymptotic independence of the numbers of high and low level crossings of stationary Gaussian processes. {\em Ann. Math. Statist.} {\bf 42}, 927--945.
\item {\sc Berman, M.S.} (1982). Sojourns and extremes of stationary processes. {\em Ann. Prob.} {\bf 10}, 1--46.
\item {\sc Berman, M.S.} (1992). \emph{Sojourns and Extremes of Stochastic Processes}, Wadsworth \& Brooks/ Cole, Boston.
\item {\sc Biau, G., Cadre, B. and Pelletier, B.} (2008). Exact rates in density support estimation. {\em J. Multivar. Anal.} {\bf 99}, 2185-2207.
\item {\sc Bickel, P. and Rosenblatt, M.} (1973a). On Some Global Measures of the Deviations of Density Function Estimates. {\em Ann. Statis.} {\bf 1}, 1071--1095.
\item {\sc Bickel, P. and Rosenblatt, M.} (1973b). Two-Dimensional Random Fields, in {\em Multivariate Analysis III, P.K. Krishnaiah, Ed.} pp. 3--15, Academic Press, New York.
\item {\sc Chen, Y-C., Genovese, C.R. and Wasserman, L.} (2013). Uncertainty Measures and Limiting Distributions for Filament Estimation. {\em arXiv: 1312.2098.v1}.
\item {\sc Chen, Y-C., Genovese, C.R. and Wasserman, L.} (2014a). Asymptotic theory for density ridges. {\em arXiv: 1406.5663}.
\item {\sc Chen, Y-C., Genovese, C.R. and Wasserman, L.} (2014b). Generalized mode and ridge estimation. {\em arXiv: 1406.1803}.
\item {\sc Chen, Y-C., Genovese, C.R. and Wasserman, L} (2015). Density level set: asymptotics, inference, and visualization. {em arXiv:1504.05438}.
\item {\sc Cheng, D.} (2015). Excursion probabilities of isotropic and locally isotropic Gaussian random fields on manifolds. {\em arXiv:1504.08047}.
\item {\sc Chernozhukov,V., Chetverikov, D. and Kato, K.} (2014). Gaussian approximation of suprema of empirical processes. {\em Ann. Statis.} {\bf 42}, 1564--1597.
\item {\sc Cram\'{e}r, H.} (1951). \emph{Mathematical Methods of Statistics}, Princeton Univ. Press, Princeton, N.J..
\item{\sc Cuevas, A., Gonz\'{a}lez-Manteiga, W. and Rodr\'{i}guez-Casal, A.} (2006). Plug-in estimation of general level sets. {\em Australian \& New Zealand Journal of Statistics} {\bf 48}, 7-19.
\item {\sc Cuevas, A. and Rodr\'{i}guez-Casal, A.} (2004). On boundary estimation. {\em Adv. Appl. Prob.} {\bf 36} 340-354.
\item {\sc de Laat, D.} (2011). Approximating Manifolds by Meshes: Asymptotic Bounds in Higher Codimension. Master Thesis, University of Groningen.
\item {\sc Genovese, C.R., Perone-Pacifico, M., Verdinelli, I. and Wasserman, L.} (2012a). The geometry of nonparametric filament estimation. {\em J. Amer. Statist. Assoc.} {\bf 107}, 788-799. 
\item {\sc Genovese, C.R., Perone-Pacifico, M., Verdinelli, I. and Wasserman, L.} (2012b). Minimax Manifold Estimation. {\em Journal of Machine Learning Research} {\bf 13}, 1263--1291.
\item {\sc Genovese, C.R., Perone-Pacifico, M., Verdinelli, I. and Wasserman, L.} (2014). Nonparametric ridge estimation.  {\em Ann. Statist.} {\bf 42},  1511-1545.
\item {\sc Gin\'{e}, E., Koltchinskii, V. and Sakhanenko, L.} (2004). Kernel density estimators: Convergence in distribution for weighted sup-norms. {\em Prob. Theory Rel. Fields} {\bf 130} 167--198.
\item {\sc Hall, P., Qian W. and Titterington, D.M.} (1992). Ridge finding from noisy data. {\em J. Comp. Graph. Statist.} {\bf 1}, 197-211.
\item {\sc H\"{u}sler J., Piterbarg, V. and Seleznjev, O.} (2003). On Convergence of the Uniform Norms for Gaussian Processes and Linear Approximation Problems. {\em Ann. Appl. Prob.} {\bf 13}, 1615--1653
\item {\sc H\"{u}sler J.} (1999). Extremes of Gaussian processes, on results of Piterbarg and Seleznjev. {\em Statist. Probab. Lett.} {\bf 44}, 251--258.
\item {\sc Koltchinskii, V., Sakhanenko, L. and Cai, S.} (2007). Integral curves of noisy vector fields and statistical problems in diffusion tensor imaging: Nonparametric kernel estimation and hypotheses testing. {\em  Ann. Statis.} {\bf 35}, 1576-1607.
\item {\sc Konakov, V.D., and Piterbarg, V.I.} (1984). On the convergence rate of maximal deviations distributions for kernel regression estimates. {\em J. Multivar. Anal.} {\bf 15}, 279--294.
\item {\sc Leadbetter, M.R. Lindgren, G. and Rootz\'{e}n, H.} (1983). \emph{Extremes and Related Properties of Random Sequences and Processes}, Series in Statistics, Springer, New York.
\item {\sc Leibon, G. and Letscher, D.} (2000). Delaunay Triangulation and Voronoi Diagrams for Riemannian Manifolds. {\em Proceedings of the Sixteenth Annual Symposium on Computational Geometry, SCG '00}, ACM New York, NY, USA, pp. 341--349.
\item {\sc Lindgren, G. and Rychlik, I} (1995). How reliable are contour curves? Confidence sets for level contours. {\em Bernoulli} {\bf 1}, 301-319.
\item  {\sc Mikhaleva, T.L. and Piterbarg, V.I.} (1997). On the Distribution of the Maximum of a Gaussian Field with Constant Variance on a Smooth Manifold. {\em Theory Prob. Appl.} {\bf 41}, 367--379.
\item {\sc Pickands, J. III.} (1969a). Asymptotic properties of the maximum in a stationary Gaussian process. {\em Trans. Amer. Math. Soc.} {\bf 145}, 55--86.
\item {\sc Pickands, J. III.} (1969b). Upcrossing Probabilities for Stationary Gaussian Processes. {\em Trans. Amer. Math. Soc.} {\bf 145}, 51--73.
\item {\sc Piterbarg, V.I.} (1996). \emph{Asymptotic Methods in the Theory of Gaussian Processes and Fields}, Translations of Mathematical Monographs, Vol. 148, American Mathematical Society, Providence, RI.
\item {\sc Piterbarg, V.I. and Stamatovich, S.} (2001). On Maximum of Gaussian Non-centered Fields Indexed on Smooth Manifolds. In {\it Asymptotic Methods in Probability and Statistics with Applications; Statistics for Industry and Technology,}\,  Eds: N. Balakrishnan, I. A. Ibragimov, V. B. Nevzorov, Birkh\"{a}user Boston, Boston, MA, pp. 189--203.
\item {\sc Qiao, W. and Polonik, W.} (2015a). Theoretical Analysis of Nonparametric Filament Estimation. {\em Submitted.}
\item {\sc Qiao, W. and Polonik, W.} (2015b). Nonparametric Confidence Regions for Density Level Sets. {\em In preparation.}
\item {\sc Rio, E.} (1994).  Local invariance principles and their applications to density estimation. {\em Prob. Theory Rel. Fields} {\bf 98}, 21--45.
\item {\sc Rosenblatt, M.} (1976). On the Maximal Deviation of $k$-Dimensional Density Estimates. {\em Ann. Prob.} {\bf 4}, 1009--1015.
\item {\sc Seleznjev, O.V.} (1991). Limit theorems for maxima and crossings of a sequence of Gaussian processes and approximation of random processes. {\em J. App. Prob.} {\bf 28}, 17--32.
\item {\sc Seleznjev, O.V.} (1996). Large deviations in the piecewise linear approximations of Gaussian processes with stationary increments. {\em Adv. Appl. Prob.} {\bf 28}, 481--499.
\item {\sc Seleznjev, O.V.} (2006). Asymptotic behavior of mean uniforms for sequences of Gaussian processes and fields. {\em Extremes} {\bf 8}, 161--169.
\item  {\sc Sharpnack J., Arias-Castro, E.} (2014). Exact asymptotics for scan statistic and fast alternatives {\em arXiv:1409.7127}.
\item {\sc Slepian, D.} (1962). The one-sided barrier problem for the Gaussian noise. {\em Bell System Tech. J.} {\bf 41}, 463--501.
\item {\sc Tan, Z.} (2015). Limit laws on extremes of non-homogeneous Gaussian random fields. {\em arXiv:1501.04422}.
\item {\sc Tan, Z., Hashorva, E. and Peng, Z.} (2012). Asymptotics of maxima of strongly dependent Gaussian processes. {\em J. Appl. Prob.} {\bf 49}, 901--1203
\end{description}

\end{section}
\end{document}